\documentclass[11pt]{amsart}

\usepackage{amsmath, amsthm, amssymb, graphicx}

\newcounter{citedtheorems}

\newtheorem{defn}{Definition}[section]
\newtheorem{theorem}[defn]{Theorem}
\newtheorem*{theorem-n}{Main Theorem}
\newtheorem*{thm-n}{Theorem}
\newtheorem*{claim-n}{Claim}
\newtheorem{thm-lit}[citedtheorems]{Theorem}
\newtheorem{defn-lit}[citedtheorems]{Definition}
\newtheorem{fact}[defn]{Fact}
\newtheorem{cor}[defn]{Corollary}

\newtheorem{concl}[defn]{Conclusion}
\newtheorem{conv}[defn]{Convention}
\newtheorem{claim}[defn]{Claim}

\newtheorem{obs}[defn]{Observation}
\newtheorem{rmk}[defn]{Remark}

\newtheorem{disc}[defn]{Discussion}

\newtheorem{qst}[defn]{Question}

\newcommand{\lost}{\L os' }
\newcommand{\los}{\L os }
\newcommand{\br}{\vspace{2mm}}

\newcommand{\nth}{\noindent\textbf}

\newcommand{\kleq}{\trianglelefteq}

\newcommand{\ml}{\mathcal{L}}
\newcommand{\tlf}{\trianglelefteq}

\newcommand{\step}{\vspace{3mm}\noindent\emph}   
\newcommand{\lp}{\emph{(}}
\newcommand{\rp}{\emph{)}}

\newcommand{\de}{\mathcal{D}}
\newcommand{\ee}{\mathcal{E}}
\newcommand{\eff}{\mathcal{F}}

\newcommand{\fss}{{\mathcal{P}}_{\aleph_0}}

\newcommand{\trv}{\textbf{t}} 
\newcommand{\uu}{\mathcal{U}}

\newcommand{\mc}{\mathcal{C}}

\newcommand{\vp}{\varphi}
\newcommand{\lcf}{\operatorname{lcf}}
\newcommand{\cf}{\operatorname{cf}}

\newcommand{\bx}{\mathbf{x}}

\newcommand{\mx}{\mathbf{x}}

\newcommand{\bh}{\mathbf{h}}

\newcommand{\oni}{\operatorname{iff}}

\title[Constructing regular ultrafilters...]{Constructing regular ultrafilters \\ from a model-theoretic point of view}
\author{M. Malliaris and S. Shelah}\thanks{\emph{Thanks}: 
Malliaris was partially supported by NSF grant DMS-1001666, Shelah's grants DMS-0600940 and 1101597, 
and by a G\"odel fellowship. Shelah was partially supported by the Israel Science Foundation grant 710/07. 
\\ This is paper 996 in Shelah's list of publications.
The authors thank Simon Thomas for very helpful organizational remarks.}
\address{Department of Mathematics, University of Chicago, 5734 S. University Avenue, Chicago, IL 60637, USA and
Einstein Institute of Mathematics, Edmond J. Safra Campus, Givat Ram, The Hebrew
University of Jerusalem, Jerusalem, 91904, Israel}
\email{mem@math.uchicago.edu}

\address{Einstein Institute of Mathematics, Edmond J. Safra Campus, Givat Ram, The Hebrew
University of Jerusalem, Jerusalem, 91904, Israel, and Department of Mathematics,
Hill Center - Busch Campus, Rutgers, The State University of New Jersey, 110
Frelinghuysen Road, Piscataway, NJ 08854-8019 USA}
\email{shelah@math.huji.ac.il}
\urladdr{http://shelah.logic.at}

\begin{document}

\begin{abstract}
This paper contributes to the set-theoretic side of understanding Keisler's order.
We consider properties of ultrafilters which affect saturation of unstable theories:
the lower cofinality $\lcf(\aleph_0, \de)$ of $\aleph_0$ modulo $\de$,  
saturation of the minimum unstable theory (the random graph), flexibility, goodness, goodness for equality, and realization of symmetric cuts. 
We work in ZFC except when noted, as several constructions appeal to complete ultrafilters thus assume a measurable cardinal.
The main results are as follows. 
First, we investigate the strength of flexibility, known to be detected by non-low theories.
Assuming $\kappa > \aleph_0$ is measurable, we construct a regular ultrafilter on $\lambda \geq 2^\kappa$ which is flexible (thus: ok) but not good, and which moreover has large $\lcf(\aleph_0)$ but does not even saturate models of the random graph. 
This implies (a) that flexibility alone
cannot characterize saturation of any theory, however (b) by separating flexibility from goodness, 
we remove a main obstacle to proving non-low does not imply maximal, and (c) from a set-theoretic point of view, consistently, 
ok need not imply good, addressing a problem from Dow 1985.
Under no additional assumptions, we prove that there is a loss of saturation in regular ultrapowers of unstable theories, 
and give a new proof that there is a loss of saturation in ultrapowers of non-simple theories. 
More precisely, for $\de$ regular on $\kappa$ and $M$ a model of an unstable theory, 
$M^\kappa/\de$ is not $(2^\kappa)^+$-saturated;
and for $M$ a model of a non-simple theory and $\lambda = \lambda^{<\lambda}$,
$M^\lambda/\de$ is not $\lambda^{++}$-saturated.
Finally, we investigate realization and omission of symmetric cuts, significant both because of the maximality of the
strict order property in Keisler's order, and by recent work of the authors on $SOP_2$.
We prove that if $\de$ is a $\kappa$-complete ultrafilter on $\kappa$, any ultrapower of a sufficiently saturated model of linear order
will have no $(\kappa, \kappa)$-cuts, and that if $\de$ is also normal, it will have a $(\kappa^+, \kappa^+)$-cut. We apply this to prove that for any $n < \omega$, assuming the existence of $n$ measurable cardinals below $\lambda$, there is a regular ultrafilter $D$ on $\lambda$ such that any $D$-ultrapower of a model of linear order will have $n$ alternations of cuts, as defined below. Moreover, 
$D$ will $\lambda^+$-saturate all stable theories but will not $(2^{\kappa})^+$-saturate any unstable theory, where $\kappa$ is the smallest
measurable cardinal used in the construction.
\end{abstract}

\maketitle

\section*{Introduction}

The motivation for our work is a longstanding, and far-reaching, problem in model theory: namely, 
determining the structure of Keisler's order on countable first-order theories. Introduced by Keisler in 1967,
this order suggests a way of comparing the complexity of first-order theories in terms of the difficulty 
of producing saturated regular ultrapowers. Much of the power of this order comes from the interplay of 
model-theoretic structure and set-theoretic constraints.   
However, this interplay also contributes to its difficulty: progress requires advances in model-theoretic
analysis on the one hand, and advances in ultrapower construction on the other. Our work in this paper 
is of the second kind and is primarily combinatorial set theory, though the model-theoretic point of view is fundamental.  

As might be expected from a problem of this scope, surprising early results were followed by 
many years of little progress.
Results of Shelah in \cite{Sh:a}, Chapter VI (1978) had settled Keisler's order for stable theories, as 
described in \S \ref{s:ex2} below. Apart from this work, and the result on maximality of $SOP_3$ in \cite{Sh500},
the problem of understanding Keisler's order on unstable theories was dormant for many years and seemed difficult.

Very recently, work of Malliaris and Shelah has led to considerable advances in the understanding of how ultrafilters and theories interact (Malliaris \cite{mm-thesis}-\cite{mm5}, Malliaris and Shelah \cite{mm-sh3}-\cite{treetops}). In particular, we now have much more information about
properties of ultrafilters which have model-theoretic significance. However, the model-theoretic analysis gave little information about the
relative strength of the ultrafilter properties described. In the current paper, we substantially clarify the picture. 
We establish various implications and non-implications between model-theoretic properties of ultrafilters, and 
we develop a series of tools and constraints which help the general problem of constructing ultrafilters with a precise degree of saturation.
Though we have framed this as a model-theoretically motivated project, it naturally relates to questions in combinatorial set theory,
and our results answer some questions there. 
Moreover, an interesting and unexpected phenomenon in this paper is the relevance of measurable cardinals in the construction of regular ultrafilters, see \ref{disc:measurable} below.

This paper begins with several introductory sections which frame our investigations and collect the implications of the current work.
We give two extended examples in \S \ref{s:examples}, the first historical, the second involving 
results from the current paper. Following this, we give definitions and fix notation in \S \ref{s:background}. 
\S \ref{s:results} gives an overview of our main results in this paper.
\S \ref{s:cont-imp} includes context for, and implications of, our constructions. Sections \S \ref{s:trg}-\S \ref{s:finite-alt} contain the main proofs. 

In this paper, we focus on product constructions and cardinality constraints.
In a related paper in preparation \cite{mm-sh-v2} we will focus on constructions via families of independent functions.

\setcounter{tocdepth}{1}

\tableofcontents \label{toc}

\newpage

\section{Background and examples} \label{s:examples}.

In this section we give two extended examples. The first is historical; we motivate the problem of Keisler's order, i.e. of classifying first-order theories in terms of saturation of ultrapowers, by explaining the classification for the stable case. 
The second involves a proof from the current paper: we motivate the idea that model-theoretic properties 
can give a useful way of calibrating the ``strength'' of ultrafilters by applying saturation arguments to prove 
that consistently flexible (=OK) does not imply good.

Some definitions will be given informally; formal versions can be found in \S \ref{s:background} below. 

\subsection{Infinite and pseudofinite sets: Theories through the lens of ultrafilters.} \label{s:ex1}
This first example is meant to communicate some intuition for the kinds of model-theoretic
``complexity'' to which saturation of ultrapowers is sensitive. 

First, recall that questions of saturation and expressive power already arise in the two fundamental theorems of ultrapowers.

\begin{thm-lit} \emph{(\lost theorem for first-order logic)} \label{t-los}
Let $\de$ be an ultrafilter on $\lambda \geq \aleph_0$, $M$ an $\ml$-structure, $\vp(\overline{x})$ an $\ml$-formula, and $\overline{a} \subseteq N = M^\lambda/\de$, $\ell({\overline{a}}) = \ell(\overline{x})$. Fixing a canonical representative of each $\de$-equivalence class, write $\overline{a}[i]$ for the value of $\overline{a}$ at index $i$.
Then 
\[ N \models \vp(\overline{a}) \iff \{ i \in \lambda : M \models \vp(\overline{a}[i]) \} \in \de \]
\end{thm-lit}

\begin{thm-lit} \emph{(Ultrapowers commute with reducts)} \label{commute-with-reducts}
Let $M$ be an $\ml^\prime$-structure, $\ml \subseteq \ml^\prime$, $\de$ an ultrafilter on $\lambda \geq \aleph_0$, $N = M^\lambda/\de$.
Then 
\[ \left( M^\lambda/\de      \right)|_{\ml}  = \left(  M|_{\ml} \right)^\lambda/\de      \]
\end{thm-lit}

That is: By itself, Theorem \ref{t-los} may appear only to guarantee that $M \equiv M^\lambda/\de$. 
Yet combined with Theorem \ref{commute-with-reducts},
it has consequences for saturation of ultrapowers, as we now explain. 

Consider the following three countable models in the language $\ml = \{ E, = \}$, for $E$ a binary relation symbol, 
interpreted as an equivalence relation.

\begin{itemize}
\item In $M_1$, $E$ is an equivalence relation with two countable classes.
\item In $M_2$, $E$ is an equivalence relation with countably many countable classes.
\item In $M_3$, $E$ is an equivalence relation with exactly one class of size $n$ for each $n \in \mathbb{N}$.
\end{itemize}

What variations are possible in ultrapowers of these models? That is, for $N_i = M_i^\lambda/\de$, what can we say about:
(a) the number of $E^{N_i}$-classes, (b) the possible sizes of $E^{N_i}$-classes, (b)$^\prime$ if two $E^{N_i}$-classes can have unequal sizes? 

\begin{obs}
For any index set $I$ and ultrafilter $\de$ on $I$,
\begin{enumerate}
\item $N_1 = (M_1)^I/\de$ will have two $E$-classes each of size $|N_1|$ 
\item $N_2 = (M_2)^I/\de$ will have $|N_2|$ $E$-classes each of size $|N_2|$ 
\end{enumerate}
\end{obs}

\begin{proof}
(1) Two classes follows by \lost theorem, so we prove the fact about size. 
By Theorem \ref{commute-with-reducts} $(M_{11}^\lambda/\de)|_\ml = M_1^\lambda/\de$, where $M_{11}$ is the expansion of $M_1$ 
to $\ml^\prime = \ml \cup \{ f \}$ and $f$ is interpreted as a bijection between the equivalence classes. By \lost theorem,
$f$ will remain a bijection in $N_1$, but Theorem \ref{commute-with-reducts} means that whether we forget the existence of $f$
before or after taking the ultrapower, the result is the same.

(2) Similarly, $M_2$ admits an expansion
to a language with a bijection $f_1$ between $M_2$ and a set of representatives of $E$-classes; a bijection $f_2$ between $M_2$ and a fixed
$E$-class; and a parametrized family $f_3(x,y,z)$ where for each $a,b$, $f_3(x,a,b)$ is a bijection between the equivalence class of $b$
and that of $b$. So once more, by Theorems \ref{t-los} and \ref{commute-with-reducts}, the ultrapowers of $M_2$ are in a sense one-dimensional: 
if $N_2 = M_2^\lambda/\de$ is an ultrapower, it will be an equivalence relation with $|N_2|$ classes each of which has size $|N_2|$.
\end{proof}

Now for $M_3$, the situation is a priori less clear. Any nonprincipal ultrapower will contain infinite (pseudofinite) sets
by \lost theorem, but it is a priori not obvious whether induced bijections between these sets exist.
It is easy to choose infinitely many distinct pseudofinite sets 
(let the $n$th set project a.e. to a class whose size is a power of the $n$th prime) 
which do not clearly admit bijections to each other in the index model $M$, nor to $M$ itself. 

We have reached the frontier of what Theorem \ref{commute-with-reducts} can control, and a property of ultrafilters comes to the surface:

\begin{defn} \label{mu-defn} \emph{(\cite{Sh:c} Definition III.3.5)}
Let $\de$ be an ultrafilter on $\lambda$.
\[  \mu(\de) :=  \operatorname{min} \left\{ \rule{0pt}{15pt}
\prod_{t<\lambda}~ n_t /\de  ~: ~ n_t < \aleph_0, ~\prod_{t<\lambda}
~n_t/\de \geq \aleph_0 \right\} \]
be the minimum value of the product of an unbounded sequence of cardinals
modulo $\de$.
\end{defn}

\begin{obs}
Let $\de$ be an ultrafilter on $\lambda$, let $M_3$ be the model defined above, and $N_3 = (M_3)^\lambda/\de$. Then:
\begin{enumerate}
\item $N_3$ will have $|N_3|$ $E$-classes.
\item $E^{N_3}$ will contain only classes of size $\geq \mu(\de)$, and will contain at least one class of size $\mu(\de)$.
\end{enumerate}
\end{obs}

\begin{proof}
(1) As for the number of classes, Theorem \ref{commute-with-reducts} still applies.

(2) Choose a sequence of cardinals $n_t$ witnessing $\mu(\de)$, and consider the class whose projection to 
the $t$th index model has cardinality $n_t$.
\end{proof}

Definition \ref{mu-defn} isolates a well-defined set-theoretic property of ultrafilters, and indeed, an early theorem of the second author
proved that one could vary the size of $\mu(\de)$:

\begin{thm-lit} \emph{(Shelah, \cite{Sh:c}.VI.3.12)} \label{mu-theorem}
Let $\mu(\de)$ be as in Definition \ref{mu-defn}. Then for any infinite
$\lambda$ and
$\nu = \nu^{\aleph_0} \leq 2^{\lambda}$ there exists a regular ultrafilter
$\de$ on $\lambda$ with $\mu(\de) = \nu$.
\end{thm-lit}

Whereas the saturation of $(M_1)^\lambda/\de$ and of $(M_2)^\lambda/\de$ will not depend on $\mu(\de)$,
$N_3 = (M_3)^\lambda/\de$ will omit a type of size $\leq\kappa$ of the form 
$\{ E(x,a) \} \cup \{ \neg x=a^\prime : N_3 \models E(a^\prime, a) \}$
if and only if $\mu(\de) \leq \kappa$. 

Restricting to regular ultrafilters, so that saturation of the ultrapower does not depend on saturation of the index model
but only on its theory, the same holds if we replace each $M_i$ by some elementarily equivalent model, and is thus a 
statement about their respective theories.

This separation of theories by means of their sensitivity to $\mu(\de)$ is, in fact, characteristic within stability.
Recall that a formula $\vp(x;y)$ has the finite cover property with respect to a theory $T$ 
if for all $n < \omega$, there are $a_0, \dots a_n$ in some model $M \models T$ such that the set
$\Sigma_n = \{ \vp(x;a_0), \dots \vp(x;a_n) \}$ is inconsistent but every $n$-element subset of $\Sigma_n$ is consistent.

\begin{thm-lit} \emph{(Shelah \cite{Sh:c} VI.5)} \label{two}
Let $T$ be a countable stable theory, $M \models T$, and $\de$ a regular ultrafilter on $\lambda \geq \aleph_0$. Then:
\begin{enumerate}
\item If $T$ does not have the finite cover property, then $M^\lambda/\de$ is always $\lambda^+$-saturated.
\item If $T$ has the finite cover property, then $M^\lambda/\de$ is $\lambda^+$-saturated if and only if $\mu(\de) \geq \lambda^+$.
\end{enumerate}
Thus Keisler's order on stable theories has exactly two classes, linearly ordered.
\end{thm-lit}

\begin{proof} (Sketch) This relies on a characterization of saturated models of stable theories: $N$ is $\lambda^+$-saturated if and only if
it is $\kappa(T)$-saturated and every maximal indiscernible set has size $\geq \lambda^+$. [This relies heavily on uniqueness of nonforking extensions:  given a type $p$ one hopes to realize over some $A$, $|A| \leq \lambda$, restrict $p$ to a small set over which it does not fork, and
use $\kappa(T)$-saturation to find a countable indiscernible sequence of realizations of the restricted type. By hypothesis, we may assume this 
indiscernible sequence extends to one of size $\lambda^+$, and by uniqueness of nonforking extensions, any element of this sequence which does
not fork with $A$ will realize the type.]

Returning to ultrapowers: for countable theories, $\kappa(T) \leq \aleph_1$ and any nonprincipal ultrapower is $\aleph_1$-saturated. So it suffices to show that any maximal indiscernible set is large, and the theorem proves, by a coding argument, that this is true whenever the size of every pseudofinite set is large.
\end{proof}

\step{Discussion}.
As mentioned above, in this paper we construct ultrafilters with ``model-theoretically significant properties.''
The intent of this example was to motivate our work by showing what ``model-theoretically significant'' might mean. 
However, the example also illustrates what kinds of properties may fit the bill. We make two general remarks.

\begin{enumerate}
\item ``Only formulas matter'': The fact that $\mu(\de)$ was detected by a property of a single formula, the finite cover property, is not an accident. 
For $\de$ a regular ultrafilter and $M \models T$ any countable theory, $M^\lambda/\de$ is $\lambda^+$-saturated if and only 
if it is $\lambda^+$-saturated for $\vp$-types, for all formulas $\vp$, Malliaris \cite{mm1} Theorem 12. Thus, from the point of view
of Keisler's order, it suffices to understand properties of regular ultrafilters which are detected by formulas. 

\item The role of pseudofinite structure is fundamental, reflecting the nature of the objects involved (regular ultrapowers, first-order theories).
On one hand, pseudofinite phenomena can often be captured by a first-order theory.
On the other, saturation of regular ultrapowers depends on finitely many conditions in each index model, since 
by definition regular ultrafilters $\de$ on $I$, $|I| = \lambda$ contain regularizing families, i.e. $\{ X_i : i < \lambda \}$
such that for each $t \in I$, $|\{ i < \lambda : t \in X_i \} | < \aleph_0$.
\end{enumerate}

\subsection{Flexibility without goodness: Ultrafilters through the lens of theories.} \label{s:ex2}

Our second example takes the complementary point of view. 
The following is a rich and important class of ultrafilters introduced by Keisler:

\begin{defn} \emph{(Good ultrafilters, Keisler \cite{keisler-1})} 
\label{good-filters}
The filter $\de$ on $I$ is said to be \emph{$\mu^+$-good} if every $f: \fss(\mu) \rightarrow \de$ has
a multiplicative refinement, where this means that for some $f^\prime : \fss(\mu) \rightarrow \de$,
$u \in \fss(\mu) \implies f^\prime(u) \subseteq f(u)$, and $u,v \in \fss(\mu) \implies
f^\prime(u) \cap f^\prime(v) = f^\prime(u \cup v)$.

Note that we may assume the functions $f$ are monotonic.

$\de$ is said to be \emph{good} if it is $|I|^+$-good.
\end{defn}

It is natural to ask for meaningful weakenings of this notion, e.g. by requiring only that certain classes of functions have multiplicative refinements. 
An important example is the notion of $OK$, which appeared without a name in Keisler \cite{keisler-1}, was named and studied by Kunen \cite{kunen-1} and investigated generally by Dow \cite{dow} and by Baker and Kunen \cite{b-k}. 
We follow the definition from \cite{dow} 1.1.

\begin{defn} \emph{(OK ultrafilters)} \label{d:ok}
The filter $\de$ on $I$ is said to be $\lambda$-OK if each monotone function $g: \fss(\lambda) \rightarrow \de$ with $g(u) = g(v)$ whenever
$|u| = |v|$ has a multiplicative refinement $f: \fss(\lambda) \rightarrow \de$.
\end{defn}

It is immediate that $\lambda^+$-good implies $\lambda$-OK.
Though OK is an a priori weaker notion, the relative strength of OK and good was not clear.  
For instance, in \cite{dow} 3.10 and 4.7, Dow raises the problem of constructing ultrafilters which are $\lambda^+$-OK but not
$\lambda^+$-good; to our knowledge, even the question of constructing $\lambda$-OK not $\lambda^+$-good ultrafilters on $\lambda$ was open.
Before discussing how a model-theoretic perspective can help with such questions, we define the main objects of interest in this paper:

\begin{defn} \emph{(Regular filters)} \label{regular}
A filter $\de$ on an index set $I$ of cardinality $\lambda$ is said to be \emph{$\lambda$-regular}, or simply 
\emph{regular}, if there exists a $\lambda$-regularizing family $\langle X_i : i<\lambda \rangle$, which means that:
\begin{itemize}
 \item for each $i<\lambda$, $X_i \in \de$, and
 \item for any infinite $\sigma \subset \lambda$, we have $\bigcap_{i \in\sigma} X_i = \emptyset$
\end{itemize}
Equivalently, for any element $t \in I$, $t$ belongs to only finitely many of the sets $X_i$. 
\end{defn}

Now we make a translation. As Keisler observed, good regular ultrafilters can be characterized as those regular ultrafilters 
able to saturate any countable theory. 
(By ``$\de$ saturates $T$'' we will always mean:
$\de$ is a regular ultrafilter on the infinite index set $I$, $T$ is a countable complete first-order theory and for any $M \models T$,
we have that $M^I/\de$ is $\lambda^+$-saturated, where $\lambda = |I|$.) 
We state this as a definition and an observation, which together say simply that
the distance between \emph{consistency} of a type (i.e. finite consistency, reflected by \lost theorem) 
and \emph{realization} of a type in a regular ultrapower
can be explained by whether or not certain monotonic functions have multiplicative refinements.

\begin{defn} \label{dist}
Let $T$ be a countable complete first-order theory, $M \models T$, $\de$ a regular ultrafilter on $I$, $|I| = \lambda$, $N = M^\lambda/\de$.
Let $p(x) = \{ \vp_i(x;a_i) : i < \lambda \}$ be a consistent partial type in the ultrapower $N$. Then a \emph{distribution} 
of $p$ is a map $d : \fss(\lambda) \rightarrow \de$ which satisfies:
\begin{enumerate}
\item For each $\sigma \in [\lambda]^{<\aleph_0}$, 
$d(\sigma) \subseteq \{ t \in I : M \models \exists x \bigwedge \{ \vp_i(x;a_i[t]) : i \in \sigma \} \}$. Informally speaking,
$d$ refines the \los map. 
\item $d$ is monotonic, meaning that $\sigma, \tau \in [\lambda]^{<\aleph_0}$, $\sigma \subseteq \tau$ implies
$d(\sigma) \supseteq d(\tau)$
\item The set $\{ d(\sigma) : \sigma \in [\lambda]^{<\aleph_0} \}$ is a regularizing family, i.e. each $t \in I$ belongs to only finitely many
elements of this set.
\end{enumerate}
\end{defn}

\begin{obs}
Let $T$ be a countable complete first-order theory, $M \models T$, $\de$ a regular ultrafilter on $\lambda$, $N = M^\lambda/\de$.
Then the following are equivalent:
\begin{enumerate}
\item For every consistent partial type $p$ in $N$ of size $\leq \lambda$, some distribution $d$ of $p$ has a multiplicative refinement.
\item $N$ is $\lambda^+$-saturated.
\end{enumerate}
\end{obs}

\begin{proof}
The obstacle to realizing the type is simply that, while \lost theorem guarantees each finite subset of $p$ is 
almost-everywhere consistent, there is no a priori reason why, at an index $t \in I$ at which 
$M \models \exists x \bigwedge \{ \vp_i(x;a_i[t]) : i \in \sigma \}$,
 $M \models \exists x \bigwedge \{ \vp_i(x;a_j[t]) : j \in \tau \}$, these two sets should have a common witness.
The statement that $d$ has a multiplicative refinement is precisely the statement that there is, in fact, a common witness
almost everywhere, in other words $t \in d(\sigma) \cap d(\tau) \implies t \in d(\sigma \cup \tau)$. When this happens,
we may choose at each index $t$ an element $c_t$ such that 
$\sigma \in [\lambda]^{<\aleph_0} \land t \in d(\sigma) \implies M \models \bigwedge \{ \vp_i(c_t : a_i[t]) : i \in \sigma \}$,
by \ref{dist}(1). Then by \lost theorem and \ref{dist}(1), $\prod_{t < \lambda} c_t$ will realize $p$ in $N$.

The other direction is clear (choose a realization $a$ and use \lost theorem to send each finite subset of the type to the 
set on which it is realized by $a$). 
\end{proof}

Keisler's characterization of good ultrafilters then follows from showing that there are first order theories which can 
``code'' enough possible patterns to detect whether any $f: \fss(\lambda) \rightarrow \de$ fails to have a multiplicative refinement.

Note that first-order theories correspond naturally to monotonic functions of a certain kind (depending, very informally speaking,
on some notion of the pattern-complexity inherent in the theory) and thus, were one to succeed in building ultrafilters which 
were able to saturate certain theories and not saturate others, this would likewise show a meaningful weakening of goodness. 

In this context we mention a property which arose in the study of certain unstable simple theories, called non-low. The
original definition is due to Buechler. 

\begin{defn}
The formula $\vp(x;y)$ is called \emph{non-low} with respect to the theory $T$ if in some sufficiently saturated model $M \models T$,
for arbitrarily large $k < \omega$, there exists an infinite indiscernible sequence $\{ a_i : i < \omega \}$, 
with $i < \omega \implies \ell(a_i) = \ell(y)$, such that every $k$-element subset of
\[ \{ \vp(x;a_i) : i < \omega \} \]
is consistent, but every $k+1$-element subset is inconsistent. 
\end{defn}

Here we make a second translation.
Recall from Definition \ref{regular} above that the characteristic objects of regular filters $\de$ on $\lambda$ are 
$(\lambda)$-\emph{regularizing families}, 
i.e. sets of the form $\{ X_i : i < \lambda \}$ with $t \in I \implies | \{ i < \lambda : t \in X_i \} | = n_t < \aleph_0$.
Malliaris had noticed in \cite{mm-thesis} that non-low formulas could detect the \emph{size} (i.e. the nonstandard integer whose
$t$th coordinate is $n_t$) of the regularizing families in $\de$, and thus had defined and studied the ``flexibility'' of a filter, Definition
\ref{flexible}. 

\begin{defn} \emph{(Flexible ultrafilters, Malliaris \cite{mm-thesis}, \cite{mm4})}
\label{flexible}
We say that the filter $\de$ is $\lambda$-flexible if for any $f \in {^I \mathbb{N}}$ with
$n \in \mathbb{N} \implies n <_{\de} f$, we can find $X_\alpha \in \de$ for $\alpha < \lambda$ such that 
for all $t \in I$
\[ f(t) \geq | \{ \alpha : t \in X_\alpha \}|  \]
Informally, given any nonstandard integer, we can find a $\lambda$-regularizing family below it.
\end{defn}

Specifically, Malliaris had shown that if $\de$ is not $\lambda$-flexible then it fails to $\lambda^+$-saturate any theory containing
a non-low formula. (Note that by Keisler's observation about good ultrafilters, \emph{any} property of ultrafilters
which can be shown to be detected by formulas must necessarily hold of good ultrafilters.) Moreover, we mention a useful convergence.  Kunen had brought the definition of ``OK'' filters to Malliaris' attention in 2010;
the notions of ``$\lambda$-flexible'' and ``$\lambda$-OK'' are easily seen to be equivalent, Observation \ref{flexible-ok} below.  

We now sketch the proof from \S \ref{s:not-flex2} below that consistently flexible need not imply good. 
(This paper and its sequel \cite{mm-sh-v2} contain at least three distinct
proofs of that fact, of independent interest.) The numbering of results follows that in \S \ref{s:not-flex2}.

To begin, we use a diagonalization argument to show that saturation decays in ultrapowers of the random graph, i.e. the Rado graph,
Definition \ref{random-graph} below.
(``The random graph'' means, from the set-theoretic
point of view, that the function which fails to have a multiplicative refinement 
will code the fact that there are two sets $A,B$ in the final ultrapower N,
$|A| = |B| = \lambda$, which are disjoint in $N$ but whose projections to the index models cannot be taken to be a.e. disjoint.)

\nth{Claim \ref{random-graph-saturation}.}
\emph{Assume $\lambda \geq \kappa \geq \aleph_0$, $T=T_{rg}$, $M$ a $\lambda^+$-saturated model of $T$, $E$ a uniform ultrafilter
on $\kappa$ such that $|\kappa^\kappa/E| = 2^\kappa$ 
Then $M^\kappa/E$ is not $(2^{\kappa})^+$-saturated.}

Note that the hypothesis of the claim will be satisfied when $E$ is regular, and also when $E$ is complete. 
Our strategy will be to take a product of ultrafilters $D \times E$, where $D$ is a regular ultrafilter on $\lambda$ and
$E$ is an ultrafilter on $\kappa$. Then $D \times E$ will be regular, and if $\lambda \geq 2^\kappa$, it will fail to saturate
the random graph, thus fail to be good. What remains is to ensure flexibility, and for this we will need $E$ to be $\aleph_1$-complete.  
In the following Corollary, $\lcf(\aleph_0, \de)$
is the coinitiality of $\mathbb{N}$ in $(\mathbb{N}, <)^I/\de$, i.e. the cofinality of the set of $\de$-nonstandard integers.

\nth{Corollary \ref{reg-transfer}}
\emph{Let $\lambda, \kappa \geq \aleph_0$ and let $\de_1$, $E$ be ultrafilters on $\lambda, \kappa$ respectively where $\kappa > \aleph_0$ 
is measurable. Let $\de = \de_1 \times E$ be the product ultrafilter on $\lambda \times \kappa$. Then:}
\begin{enumerate}
\item \emph{If $\de_1$ is $\lambda$-flexible and $E$ is $\aleph_1$-complete, then $\de$ is $\lambda$-flexible.}
\item \emph{If $\lambda \geq \kappa$ and $\lcf(\aleph_0, \de_1) \geq \lambda^+$, then $\lcf(\aleph_0, \de) \geq \lambda^+$, so in particular,
$\de = D \times E$ will $\lambda^+$-saturate any countable stable theory.}
\end{enumerate}

\begin{proof} (Sketch)
(1) We first show that the following are equivalent: (i) any $\de$-nonstandard integer projects $E$-a.e. to a $\de_1$-nonstandard integer,
(ii) $E$ is $\aleph_1$-complete. Then, since we have assumed (ii) holds, let some $\de$-nonstandard integer $n_*$ be given.
By (ii), for $E$-almost all $t \in \kappa$, $n_*[t]$ is $\de_1$-nonstandard and by the flexibility of $\de_1$
there is a regularizing family $\{ X^t_i : i < \lambda \} \subseteq \de_1$ below any such $n_*[t]$. 
Let $X_i = \{ (s,t) : s \in X^t_i \} \subseteq \de$.  
It follows that $\{ X_i : i <\lambda \}$ is a regularizing family in $\de$ below $m_*$ and thus below $n_*$.

(2) From the first sentence of (1), we show that if $E$ is $\aleph_1$-complete and in addition 
$\lcf(\aleph_0, \de) \geq \lambda^+$, then the $\de_1$-nonstandard integers (under the diagonal embedding)
are cofinal in the $\de$-nonstandard integers. This suffices. 
For the second clause, see Theorem \ref{formula-corr}, \S \ref{s:cont-imp} below.
\end{proof}

Thus we obtain:

\nth{Theorem \ref{flexible-not-good-2}.}
\emph{Assume $\aleph_0 < \kappa < \lambda = \lambda^\kappa$, $2^\kappa \leq \lambda$, $\kappa$ measurable.
Then there exists a regular uniform ultrafilter $\de$ on $\lambda$ such that $\de$ is $\lambda$-flexible, yet
for any model $M$ of the theory of the random graph, $M^\lambda/\de$ is not $(2^\kappa)^+$-saturated. Thus
$\de$ is not good, and will fail to $(2^\kappa)^+$-saturate any unstable theory. However,
$\de$ will $\lambda^+$-saturate any countable stable theory.}

Note that the model-theoretic failure of saturation is quite strong, more so than simply ``not good.'' 
The random graph is known to be minimum among 
unstable theories in Keisler's order (meaning that any regular $\de$ which fails to saturate the random graph will fail to saturate any
other unstable theory). This is the strongest failure of saturation one could hope for given that $\lcf(\aleph_0, \de)$ is large,
see Section \ref{s:cont-imp} for details. 

Theorem \ref{flexible-not-good-2} has the following immediate corollary in the language of OK and good:

\begin{cor}
Assume $\aleph_0 < \kappa < 2^\kappa \leq \mu_1 \leq \mu_2 < \lambda = \lambda^\kappa$ and $\kappa$ is measurable.
Then there exists a regular uniform ultrafilter $\de$ on $\lambda$ such that $\de$ is $\lambda$-flexible, thus $\lambda$-OK, 
but not $(2^\kappa)^+$-good.
In particular, $\de$ is $(\mu_2)^+$-OK but not $(\mu_1)^+$-good. 
\end{cor}

In particular, this addresses the problem raised by Dow in \cite{dow} 3.10 and 4.7,
namely, the problem of constructing ultrafilters which are $\alpha^+$-OK and not $\alpha^+$-good.
\br

\step{Discussion}. The intent of this example was to show that model theory can contribute to calibrating ultrafilters.
Note that in terms of determining the strength of a priori weakenings of goodness, 
the model-theoretic perspective has given both positive and negative results:
\begin{enumerate}
\item  On one hand, Theorem \ref{flexible-not-good-2} applies model-theoretic arguments to show that multiplicative refinements for 
size-uniform functions $f: \fss(\lambda) \rightarrow \de$ are not enough to guarantee multiplicative refinements for all such functions.
\item On the other, the second author's proof of the maximality of strict order (see Theorem \ref{formula-corr}(6) below) \emph{does}
isolate an a priori weaker class of functions which have such a guarantee -- namely, those corresponding 
to distributions of types in linear order. The set-theoretic question of why these functions suffice appears to be deep. 
The model-theoretic formulation of ``determine a minimum such set of functions'' is: determine a necessary condition for 
maximality in Keisler's order. 
\end{enumerate}

This concludes our two examples. We now fix definitions and notation, before giving a summary of our results in \S \ref{s:results}.

\section{Definitions and conventions} \label{s:background}
This section contains background, most definitions, and conventions. Note that the definition of $\mu(\de)$ was given in Definition \ref{mu-defn},
and the definitions of good, regular and flexible filters were Definitions \ref{good-filters}, \ref{regular} and \ref{flexible} above.
(Recall that a filter is said to be \emph{$\lambda$-regular} if it contains a family of $\lambda$ sets 
any countable number of which have empty intersection, \ref{regular} above.)

Let $I = \lambda \geq \aleph_0$ and fix $f: \fss(\lambda)\rightarrow I$. Then 
$\{ \{ s\in I : \eta \in f^{-1}(s) \} : \eta < \lambda \}$ can be extended to a regular filter on $I$,
so regular ultrafilters on $\lambda \geq \aleph_0$ always exist, see \cite{ck73}.

Keisler proposed in 1967 \cite{keisler} that saturation properties of regular ultrapowers might be 
used to classify countable first-order theories. 
His preorder $\kleq$ on theories is often thought of as a partial order on the $\kleq$-equivalence classes, and so known as
``Keisler's order.'' 

\begin{defn} \label{keisler-order} \emph{(Keisler \cite{keisler})} 
Given countable theories $T_1, T_2$, say that:
\begin{enumerate}
 \item  $T_1 \kleq_\lambda T_2$ if for any $M_1 \models T_1, M_2 \models
T_2$, and $\de$ a regular ultrafilter on $\lambda$, \\if 
$M^{\lambda}_2/\de$
is $\lambda^+$-saturated then $M^{\lambda}_1/\de$ must be 
$\lambda^+$-saturated.
\item \emph{(Keisler's order)}
$T_1 \kleq T_2$ if for all infinite $\lambda$, $T_1 \kleq_\lambda T_2$.
\end{enumerate}
\end{defn}

\begin{qst}
Determine the structure of Keisler's order.
\end{qst}

The hypothesis \emph{regular} justifies the quantification over all models: when $T$ is countable and $\de$ is regular, saturation of the ultrapower does not depend on the choice of index model. 

\begin{thm-lit} \label{backandforth} \emph{(Keisler \cite{keisler} Corollary 2.1 p. 30; see also Shelah \cite{Sh:c}.VI.1)}
Suppose that $M_0 \equiv M_1$, the ambient language is countable, and
$\de$ is a regular ultrafilter on $\lambda$.
Then ${M_0}^\lambda/\de$ is $\lambda^+$-saturated iff ${M_1}^\lambda/\de$
is $\lambda^+$-saturated.
\end{thm-lit} 

More information on Keisler's order, including many examples and a summary of all known results through early 2010, 
may be found in the introduction to the first author's paper \cite{mm4}. 
More recent results are incorporated (in a somewhat denser form) into the various summary theorems of \S \ref{s:cont-imp} below.

On the opposite end of the spectrum to ``regular,'' we have:

\begin{defn} \emph{(Complete ultrafilters)}
The ultrafilter $\ee$ on $\kappa$ is said to be $\mu$-complete if for any $\{ X_i : i < \mu^\prime < \mu \} \subseteq \ee$,
$\bigcap \{ X_i : i < \mu^\prime \} \in \ee$.
\end{defn}

Working with complete ultrafilters, we are obliged to make large cardinal hypotheses. We will use measurable, normal and to a lesser
extent, weakly compact cardinals.
Their utility for our arguments will be clear from the choice of definitions:

\begin{defn} \emph{(Measurable, weakly compact)}
\begin{enumerate}
\item The uncountable cardinal $\kappa$ is said to be \emph{measurable} if there is a $\kappa$-complete nonprincipal ultrafilter on $\kappa$.
\item The uncountable cardinal $\kappa$ is said to be \emph{weakly compact} if $\kappa \rightarrow (\kappa)^2_2$. 
\end{enumerate}
\end{defn}

\begin{fact} \label{wcc-fact}
If $\kappa > \aleph_0$ is weakly compact, $n < \aleph_0$ and $\rho < \kappa$, then for any $\alpha: [\kappa]^n \rightarrow \rho$
there exists $\uu \subseteq \kappa$, $|\uu| = \kappa$ such that
$\langle \alpha(\epsilon_1, \dots \epsilon_n) : \epsilon_1 < \dots < \epsilon_n ~\mbox{from $\uu$} \rangle$ is constant.
\end{fact}

\begin{defn} \label{d:normal} \emph{(Normal ultrafilters)} 
A filter $D$ on $\kappa$ is \emph{normal} when, for any sequence
$\langle A_i : i < \kappa \rangle$ with $i < \kappa \implies A_i \in D$, 
\[ \{ \alpha < \kappa : (\forall j < 1+\alpha) (\alpha \in A_j) \}  \in D \]
\end{defn}

\begin{fact} \label{normal-fodor}
Let $\kappa$ be a measurable cardinal. Then
\begin{enumerate}
\item there exists a normal, $\kappa$-complete, uniform ultrafilter $D$ on $\kappa$.
\item for any $f: \kappa \rightarrow \kappa$ which is regressive on $X \in D$, there is a set $Y \in D$, $Y \subseteq X$
on which $f$ is constant.
\end{enumerate}
\end{fact}

\begin{disc} \label{disc:measurable}
An interesting and unexpected phenomenon visible in this work is the relevance of measurable cardinals, and in particular
$\kappa$-complete nonprincipal ultrafilters, in the construction of regular ultrafilters. 
In the 1960s, model theorists pointed out regularity as a central property of ultrafilters, and generally concentrated
on this case. Regularity ensures that saturation 
[of ultrapowers of models of complete countable theories] does not depend on the saturation of the index model, and that 
the cardinality of ultrapowers can be settled ($M^I/\de = M^{|I|}$). Meanwhile, the construction of various non-regular ultrafilters
was investigated by set theorists. However, many questions about regular ultrafilters remained opaque from the model-theoretic point of 
view. For example, from the point of view of Theorem \ref{two}, p. \pageref{two} above, 
the regular ultrafilters with large $\lcf(\aleph_0)$ -- a condition which implies that
these ultrafilters saturate ultrapowers of stable theories -- appeared to look alike. Moreover, 
it was not clear whether various a priori weakenings of goodness (e.g. flexible/ok) were indeed weaker. 
Here, in several different constructions, we combine both lines of work, using 
$\kappa$-complete ultrafilters to construct regular ultrafilters on $\lambda > \kappa$ with model-theoretically meaningful
properties, i.e. presence or absence of some specific kind of saturation. 
\end{disc}

We mention several other relevant properties.

\begin{defn} \emph{(Good for equality, Malliaris \cite{mm5})}
Let $\de$ be a regular ultrafilter. Say that $\de$ is \emph{good for equality}
if for any set $X \subseteq N = M^I/\de$,
$|X| \leq |I|$, there is a distribution $d: X \rightarrow \de$ such that $t \in \lambda, t \in d(a) \cap d(b)$ implies that 
$(M \models a[t] = b[t]) \iff (N \models a = b)$.
\end{defn}

\begin{defn} \emph{(Lower cofinality, $\lcf(\kappa, D)$)} \label{lcf}
Let $D$ be an ultrafilter on $I$ and $\kappa$ a cardinal. Let $N = (\kappa, <)^I/\de$.
Let $X \subset N$ be the set of elements above the diagonal embedding of $\kappa$.
We define $\lcf(\kappa, D)$ to be the cofinality of $X$ considered with the reverse order.
\end{defn}

Note: this is sometimes called the coinitiality of $\kappa$.

\begin{defn} \emph{(Product ultrafilters)} \label{d:prod}
Let $I_1, I_2$ be infinite sets and let $D_1, D_2$ be ultrafilters on $I_1, I_2$ respectively.
Then the product ultrafilter $D = D_1 \times D_2$ on $I_1 \times I_2$ is defined by:
\[ X \in D \iff \{ t \in I_2 ~:~ \{ s \in I_1 ~:~ (s,t) \in X \} \in D_1 \} \in D_2 \]
for any $X \subseteq I_1 \times I_2$.
\end{defn}

Finally, it will be useful to have a name for functions, or relations, to which \lost theorem applies since 
they are visible in an expanded language:

\begin{defn} \label{induced} \emph{(Induced structure)}
Let $N = M^\lambda/\de$ be an ultrapower and $X \subseteq N^{m}$. Say that $X$ is an \emph{induced} function, or relation, if there exists
a new function, or relation, symbol $P$ of the correct arity, and an expansion $M^\prime_t$ of each index model $M_t$ to $\ml \cup \{ P \}$,
so that $P^N \equiv X \mod \de$.

Equivalently, $X$ is the ultraproduct modulo $\de$ of its projections to the index models.
\end{defn}

\begin{defn} \label{mc} \emph{(Cuts in regular ultrapowers of linear orders)}
\begin{enumerate}
\item
For a model $M$ expanding the theory of linear order, a $(\kappa_1, \kappa_2)$-cut in $M$ is given by
sequences $\langle a_i : i < \kappa_1 \rangle$, $\langle b_j : j < \kappa_2 \rangle$ of elements of $M$ such that
\begin{itemize}
\item $i_1 < i_2 < \kappa_1 \implies a_{i_1} < a_{i_2}$
\item $j_1 < j_2 < \kappa_2 \implies b_{j_2} < b_{j_1}$
\item $i < \kappa_1$, $j < \kappa_2$ implies $a_i < b_j$  and
\item the type $\{ a_i < x < b_j : i < \kappa_1, j<\kappa_2 \}$ is omitted in $M$.
\end{itemize} 

\item For $\de$ a (regular) ultrafilter on $I$ we define: 
\[ \mc(\de) = \left\{ (\kappa_1, \kappa_2) \in (\operatorname{Reg} \cap {|I|}^+) \times (\operatorname{Reg} \cap {|I|}^+) : 
 ~\mbox{$(\mathbb{N}, <)^I/\de$ has a $(\kappa_1, \kappa_2)$-cut} \right\} \]
\end{enumerate}
\end{defn}

Here we list the main model-theoretic properties of formulas used in this paper. For $TP_1$/$SOP_2$ and $TP_2$, 
see \S \ref{s:non-simple}. The finite cover property is from Keisler \cite{keisler} and the
order property, independence property and strict order property are from Shelah \cite{Sh:c}.II.4.

\begin{defn} \emph{(Properties of formulas)} \label{properties-of-formulas}
Let $\vp=\vp(x;y)$ be a formula of $T$ and $M \models T$ be any sufficiently saturated model.
Note that $\ell(x), \ell(y)$ are not necessarily 1.
Say that the formula $\vp(x;y)$ has:

\begin{enumerate}
\item \emph{not the finite cover property}, written \emph{nfcp}, if there exists $k<\omega$ such that:
for any $A \subseteq M$ and any set $X = \{ \vp(x;a) : a \in A \}$ of instances of $\vp$, $k$-consistency implies consistency.
(This does not depend on the model chosen.)
\item the \emph{finite cover property}, written \emph{fcp}, if it does not have nfcp: that is, for arbitrarily large
$k < \omega$ there exist $a_0, \dots a_k \in M$ such that $\{ \vp(x;a_0), \dots \vp(x;a_k) \}$ is inconsistent,
but every $k$-element subset is consistent.
\item the \emph{order property} if there exist elements $a_i$ $(i<\omega)$ such that
for each $n<\omega$, the following partial type is consistent: 
\[ \{ \neg \vp(x;a_i) : i < n \} \cup \{ \vp(x;a_j) : j \geq n \} \]
\\ Formulas with the order property are called $\emph{unstable}$. 
\item the \emph{independence property} if there exist elements $a_i$ $(i<\omega)$ such that
for each $\sigma, \tau \in [\omega]^{<\aleph_0}$ with $\sigma \cap \tau = \emptyset$, the following partial type is consistent:
\[ \{ \neg \vp(x;a_i) : i \in \sigma \} \cup \{ \vp(x;a_j) : j \in \tau \} \]
\\ Note that the independence property implies the order property.
\item the \emph{strict order property} if there exist elements $a_i$ $(i<\omega)$ such that for all $j \neq i < \omega$,
\[ \left( \exists x ( \neg \vp(x;a_j) \land \vp(x;a_i)) \iff j<i \right) \]
\\ Note that $(4)$, $(5)$ each imply $(3)$.
\end{enumerate}

\noindent A theory $T$ is said to have the finite cover property, the order property, the independence property or
the strict order property 
iff one of its formulas does, and to have $nfcp$ if all of its formulas do.
\end{defn}

The ``random graph'' is known to be minimum in Keisler's order among the unstable theories, and so will feature in our proofs
with some regularity. 

\begin{defn} \label{random-graph}
The \emph{random graph}, i.e. the Rado graph, is (the unique countable model of) 
the complete theory in the language with equality and a binary relation 
$R$ axiomatized by saying that there are infinitely many elements, and that for each $n$, and any two disjoint subsets of size $n$,
there is an element which $R$-connects to all elements in the first set and to none in the second set. 
\end{defn}

We conclude this section with some conventions which hold throughout the paper.

\begin{conv} \emph{(Conventions)} \label{conventions}
\begin{itemize}
\item
The letters $D, E, \de, \ee$ are used for filters. 
Generally, we reserve $\de$ for a regular filter or ultrafilter, 
and $\ee$ for a $\kappa$-complete ultrafilter where $\kappa \geq \aleph_0$, though this is always stated in the relevant proof. 
\item Throughout, tuples of variables may be written without overlines, that is: 
when we write $\vp = \vp(x;y)$, neither $x$ nor $y$ are necessarily assumed to have length 1.
\item For transparency, all languages are assumed to be countable. 
\item As mentioned in \S \ref{s:ex2}, by ``$\de$ saturates $T$'' we will always mean:
$\de$ is a regular ultrafilter on the infinite index set $I$, $T$ is a countable complete first-order theory and for any $M \models T$,
we have that $M^I/\de$ is $\lambda^+$-saturated, where $\lambda = |I|$.
\item We will also say that the ultrafilter $D$ is ``good'' \lp or: ``not good''\rp~ for the theory $T$ to mean that $D$ saturates 
\lp or: does not saturate\rp~ the theory $T$.
\item We reserve the word \emph{cut} in models of linear order for omitted types. A partial type in a model $M$ given by some
pair of sequences $( \langle a_\alpha : \alpha < \kappa_1 \rangle, \langle b_\beta : \beta < \kappa_2 \rangle)$ with
$\alpha < \alpha^\prime < \kappa_1, \beta < \beta^\prime < \kappa_2 \implies M \models a_\alpha < a_{\alpha^\prime} < b_{\beta^\prime} < b_\beta$,
which may or may not have a realization in $M$, is called a \emph{pre-cut}. See also Definition \ref{mc}.
\end{itemize}
\end{conv}

\section{Description of results} \label{s:results}

In this section, we describe the main results of the paper. 

Some notes: For relevant definitions and conventions (``$\de$ saturates $T$,'' ``good for,'' ``pre-cut'') see \S \ref{s:background} above,
in particular \ref{conventions}. Lists of the properties mentioned from the point of view of Keisler's order can be found in
Theorems \ref{formula-corr}-\ref{theory-corr}, \S \ref{s:cont-imp}.
The reader unused to phrases of the form ``not good for the random graph therefore not good'' 
is referred to \S \ref{s:ex2}, in particular the half-page following Definition \ref{regular}.

\br
The first main result, Theorem \ref{flexible-not-good-2}, was discussed in \S \ref{s:ex2} above; we will give a few more details here. 
We begin by showing that
for $E$ an ultrafilter on $\kappa$, if $E$ is $\kappa$-regular or $E$ is $\kappa$-complete then $M^\kappa/E$ 
will not be $(2^\kappa)^+$-saturated for any $\kappa^+$-saturated model of $T_{rg}$, the theory of the random graph. This is essentially
a diagonalization argument, and the saturation of $M$ is not important if $E$ is regular. Note that with this result in hand, 
we have a useful way of producing regular ultrafilters which are not good: for any $\lambda$-regular $\de_1$ on $\lambda$, 
the product ultrafilter $\de = \de_1 \times E$ will remain regular but will not be good for $T_{rg}$ when $2^\kappa \leq \lambda$. 
We then show, as sketched above, that
when $E$ is $\aleph_1$-complete and $\lcf(\aleph_0, \de_1)$ is large the nonstandard integers of $\de_1$ will be cofinal in those
of $\de$ (under the diagonal embedding) and thus that $\de$ will inherit both the large $\lcf(\aleph_0)$ and the flexibility of $\de_1$. 

This yields Theorem \ref{flexible-not-good-2}: for any $\lambda \geq 2^\kappa$, $\kappa$ measurable there is a regular 
ultrafilter on $\lambda$ which is flexible but not good, which 
has large $\lcf(\aleph_0)$ and thus saturates all stable theories, but does not saturate any unstable theory. 
In at least one sense, this is a surprising reversal. 
From the model-theoretic point of view, flexible ultrafilters had appeared
``close'' in power to ultrafilters capable of saturating any first-order theory. By Theorem \ref{flexible-not-good-2},
consistently flexibility cannot guarantee the saturation of any unstable theory, since the random graph is minimum among the unstable theories
in Keisler's order. Thus the space between flexible and good is potentially quite large.

\emph{Discussion.}
In fact we obtain several different flexible-not-good theorems, including Conclusion \ref{flex-not-good-tp2} and
a related theorem in \cite{mm-sh-v2}. These results have other advantages, and less dramatic failures of saturation.
In fact, if the stated failures of saturation can be shown to be sharp, this could be quite useful for obtaining further dividing lines within the unstable theories.

Second, we give a new proof that there is a loss of saturation in ultrapowers of non-simple theories, Conclusion \ref{simple-non-sat}. 
Specifically, we show that if
$M$ is a model of a non simple theory and $\de$ a regular ultrafilter on $\lambda = \lambda^{<\lambda}$, then for some formula $\vp$
$M^\lambda/\de$ is not $\lambda^{++}$-saturated for $\vp$-types. (The second author's book \cite{Sh:c}.VI contains a first proof of this result.) 
Unlike the unstable case, we do not need diagonalization per se,
but instead show that both forms of the tree property allow for large consistent types to be built from minimally compatible constellations
of formulas in each (quite saturated) index model. Since a degree of inconsistency just above what is visible to \lost theorem has been hard-coded in
to the choice of type, and the size of the type is larger than the size of the index set, we can use Fodor's lemma to push any possible
realization into one of the index models to obtain a contradiction. 

The remainder of the paper focuses on realization and omission of symmetric cuts, and here complete ultrafilters are very useful;
see Discussion \ref{disc:measurable} above. 
This work is best understood in the context of two facts which highlight the importance of linear order for calibrating saturation.
First, a surprising early result on Keisler's order was the second author's proof
that any theory with the strict order property is maximal. 
Second, some key properties of ultrafilters which reflect much lower levels of saturation such as $\mu(\de)$, $\lcf(\aleph_0, \de)$ 
(see the comprehensive theorems of Section \ref{s:cont-imp}) are detected by ultrapowers of linear order.
Thus, information about a given ultrafilter's ability to 
realize or omit types in linear order is, in principle, very informative.
The focus on symmetric cuts is underlined by a recent result of the authors in \cite{treetops} that any regular ultrafilter which saturates some theory 
with $SOP_2$ (Definition \ref{sop2-tree} below) will realize all symmetric pre-cuts. 
Thus omission of symmetric cuts means the ultrafilter is not good for a theory which is not known to be maximal.  

Returning to the synopsis of results,
we prove in Claim \ref{complete-cuts1} that 
if the ultrafilter $\ee$ is $\kappa$-complete not $\kappa^+$-complete, any $\ee$-ultrapower of a sufficiently saturated model of linear order
will have no $(\kappa, \kappa)$-cuts. This is a fairly direct proof, and we also show there that if we further assume that $\ee$ is normal then
it follows that $\ee$ is good (though not regular: see the Appendix e.g. \ref{fact-good}).
We then prove, in Claim \ref{kappa-plus-cut}, that if $\ee$ is $\kappa$-complete and normal on $\kappa$, then 
any ultrapower of a sufficiently saturated model of linear order will contain a $(\kappa^+, \kappa^+)$-cut. 
Here the proof is to build the cut in question. We work by induction on $\alpha < \kappa^+$, and the intent is on one hand, to define
a pair of elements $(f_\alpha, g_\alpha)$ which, in the ultrapower, continue the current pre-cut; and on the other,
to ensure globally that the sequence of such intervals is a type with no realization. To accomplish the second, informally speaking, 
we make the inductive choice of $(f_\alpha(\epsilon), g_\alpha(\epsilon))$ at index $\epsilon$ to be compatible with a relatively short, linearly ordered segment of the cut so far, but no more. We then use normality to show that over the length of the index set, an interval chosen 
in this way will in fact continue to define a pre-cut. However, at the end of the construction, this enforced bound on compatibility 
will prevent realization of the type, since by Fodor's lemma any potential realization of the cut will reflect the existence of a long
linearly ordered set in some index model. 

Finally, we leverage these proofs for a last existence result, Theorem \ref{alt-cuts}.
There, assuming $n$ measurable cardinals below $\lambda$, we show how to take products of 
ultrafilters to produce any number of finite alternations of cuts in an ultrafilter on $\lambda$ with clearly described saturation properties. 
The proof is by downward induction. Note that if the first ultrafilter in our
$(n+1)$-fold product is regular, the final ultrafilter will be regular. By a result mentioned at the beginning of this section, 
if the first ultrafilter is flexible with large $\lcf(\aleph_0)$ and all remaining ultrafilters are at least $\aleph_1$-complete, the final 
ultrafilter will inherit flexibility and large $\lcf(\aleph_0)$, thus be able to saturate any stable theory. By the result about
loss of saturation for the random graph, it will fail to $(2^\kappa)^+$-saturate any unstable theory, where $\kappa$ is the smallest
measurable cardinal used in the construction. 

This completes the summary of our main results. 
In the appendix, we collect some easy observations and extensions of previous results. We discuss goodness for non-regular ultrafilters,
give the equivalence of flexible and OK, and extend the second author's proof from \cite{Sh:c}.VI.4.8, 
that an $\aleph_1$-incomplete ultrafilter with small $\lcf(\aleph_0, D)$ fails to saturate unstable theories, 
to the case of $\kappa$-complete, $\kappa^+$-incomplete ultrafilters with small $\lcf(\kappa, D)$.

\section{Equivalences and implications} \label{s:cont-imp}

This section justifies the phrase ``properties of ultrafilters with model-theoretic significance.'' 
We state and prove several comprehensive theorems which give the picture of Keisler's order, Definition \ref{keisler-order},
in light of our current work and of the recent progress mentioned in the introduction.
 
Recall our conventions in \ref{conventions}, especially with respect to ``saturates'' and ``good''. 
Minimum, maximum, etc. refer to Keisler's order.

The first theorem collects the currently known correspondences between properties of regular ultrafilters and properties of 
first-order theories.

\begin{thm-lit} \label{formula-corr}
In the following table, for each of the rows \emph{(1),(3),(5),(6)} the regular ultrafilter $\de$ on $\lambda$ fails to have 
the property in the left column if and only if it omits a type in every formula with the property in the right column.
For rows \emph{(2)} and \emph{(4)}, left to right holds: if $\de$ fails to have the property on the left then it omits a type 
in every formula with the property on the right.

\br
\begin{tabular}{lcl}
\br \textbf{Set theory: properties of filters} & & \textbf{Model theory: properties of formulas} \\
 \emph{(1)} $\mu(\de) \geq \lambda^+$ & & A. finite cover property \\
 \emph{(2)} $\lcf(\aleph_0, \de) \geq \lambda^+$ & ** & B. order property \\
 \emph{(3)} good for $T_{rg}$ & & C. independence property \\
 \emph{(4)} flexible, i.e. $\lambda$-flexible	& ** & D. non-low \\
 \emph{(5)} good for equality & & E. $TP_2$ \\
 \emph{(6)} good, i.e. $\lambda^+$-good&  & F. strict order property \\ 
\end{tabular}
\end{thm-lit}

\begin{proof} (Discussion - Sketch)

(1) $\leftrightarrow$ (A) Shelah \cite{Sh:c}.VI.5, see \S \ref{s:ex1} above.

(2) $\leftarrow$ (B) Shelah \cite{Sh:c}.VI.4.8, see also Theorem \ref{lcf-omega} below which generalizes that result.

(3) $\leftrightarrow$ (C) Straightforward by quantifier elimination, see \cite{mm4}. 
More generally, Malliaris \cite{mm5} shows that the random graph, as the minimum non-simple theory, and
$T_{feq}$, as the minimum $TP_2$ theory, are in a natural sense characteristic of ``independence properties'' seen by ultrafilters. 

(4) $\leftarrow$ (D) Malliaris \cite{mm-thesis}, see \S \ref{s:ex2} above, or \cite{mm-sh-v2}.

(5) $\leftrightarrow$ (E) Malliaris \cite{mm4} \S 6, which proves the existence of a minimum $TP_2$-theory, the theory $T^*_{feq}$ of
a parametrized family of independent (crosscutting) equivalence relations. 

(6) $\leftrightarrow$ (F) Keisler observed that good ultrafilters can saturate any countable theory, and 
proved that goodness is equivalent to the saturation of certain (``versatile'') formulas \cite{keisler},
thus establishing the existence of a maximum class in Keisler's order; see \S \ref{s:ex2} above.
The result (6) $\leftrightarrow$ (F) follows from Shelah's proof in \cite{Sh:c}.VI.3 that any theory with the strict order property is maximum
in Keisler's order. Thus any ultrafilter able to saturate $SOP$-types must be good, and by Keisler's observation the reverse holds.
In fact, $SOP_3$ is known to be sufficient for maximality by \cite{Sh500}-\cite{ShUs}, but this 
formulation is more suggestive here given our focus on order-types and cuts. {A model-theoretic characterization of the maximum
class is not known.}
\end{proof}

\begin{rmk}
Moreover, by work of the authors in \cite{treetops} if $\de$ on $\lambda$ has ``treetops,'' i.e. it realizes a certain set of $SOP_2$-types then it
must realize all symmetric pre-cuts, that is, there can be no $(\kappa, \kappa)$-cuts in ultrapowers of linear order for $\kappa \leq \lambda$. 
So we will also be interested in the property of realizing symmetric cuts. 
\end{rmk}

As rows (2) and (4) of Theorem \ref{formula-corr} suggest, there are subtleties to the correspondence. 
If $T$ is not Keisler-maximal 
then any formula $\vp$ of $T$ with the order property has the independence property, as does any non-low formula. Yet consistently neither (4) nor (2)
imply (3), as the rest of this section explains.
So while we have model-theoretic sensitivity to properties (2) and (4),   
this is not enough for a characterization: in fact it follows from the theorems below that there is consistently no theory
(and no formula $\vp$) which is saturated by $\de$ if and only if (2), or (4).

\begin{thm-lit} \label{theory-corr}
Using the enumeration of properties of ultrafilters from Theorem \ref{formula-corr}, we have that: 
\begin{enumerate}
\item
is necessary and sufficient for saturating stable theories, 
\item  
is necessary for saturating unstable theories, 
\item 
is necessary and sufficient for saturating the minimum unstable theory, 
\item 
is necessary for saturating non-low theories, 
\item 
is necessary and sufficient for saturating the minimum $TP_2$ theory, 
\item 
is necessary and sufficient for saturating any Keisler-maximum theory, e.g.
$Th(\mathbb{Q}, <)$; \newline note that the identity of the maximum class is not known. 
\end{enumerate}
\end{thm-lit}

\begin{proof}[Discussion]
The sources follow those of Theorem \ref{formula-corr}, but we make some additional remarks. 

(1) Note that Shelah's proof of (1) in \cite{Sh:c}.VI.5, quoted and sketched as Theorem \ref{two}, \S \ref{s:ex1} above, gives the only two known equivalence classes in Keisler's order. 

(2)  By Shelah \cite{Sh:c}.VI.4.8 (or Theorem \ref{lcf-omega} below) if $\de$ is regular and $\lcf(\de, \aleph_0) \leq \lambda^+$ then any $\de$-ultrapower will omit a $\lambda$-type in some unstable formula, i.e., a formula with the order property. From the set-theoretic point of view, (2) $\not\rightarrow$ (3) of Theorem \ref{implications} shows that $\lcf(\aleph_0, \de) \geq \lambda^+$ is weaker than
ensuring $\lambda^+$-saturation for the random graph (or equivalently, for some formula with the independence property). 
From the model-theoretic point of view, since any unstable theory has either 
the strict order property or the independence property, this gap is not visible.

(3), (5) In fact what the characterization in Malliaris \cite{mm4} shows is that a necessary and sufficient condition for an ultrafilter $\de$ on $\lambda$ to saturate the minimum $TP_2$ theory is that it be ``good for equality,'' meaning that for any set $X \subseteq N = M^\lambda/\de$,
$|X| = \lambda$, there is a distribution $d: X \rightarrow \de$ such that $t \in \lambda, t \in d(a) \cap d(b)$ implies that 
$(M \models a[t] = b[t]) \iff (N \models a = b)$. By contrast, saturation of the minimum unstable theory asks only that for any two disjoint sets
$X, Y \subseteq M^\lambda/\de$, $|X| = |Y| = \lambda$, there is a distribution $d: X \cup Y \rightarrow \de$ such that for any $t \in \lambda$,
$a \in X, b \in Y$, $t \in d(a) \cap d(b)$ implies $M \models a[t] \neq b[t]$. 

(4) This was discussed in \S\ref{s:ex2} above.
Note that by work of Shelah \cite{Sh93} non-simple theories have an inherent structure/randomness ``dichotomy'' of $TP_1$ versus $TP_2$, analogous
to the structure/randomness dichotomy for unstable theories of $SOP$ versus $IP$; see \S \ref{s:non-simple} below.
We know from \cite{mm4} 
that any ultrafilter which saturates the minimum $TP_2$-theory must be flexible, however we do not know whether an
ultrafilter which saturates some $SOP_2$ theory must be flexible. 

(6) See the proof of Theorem \ref{formula-corr}(6).
\end{proof}

With the progress in this paper, and in other recent work of the authors, 
we may summarize the current picture of implications as follows:

\begin{theorem}  
\label{implications}
Let properties $(1)-(6)$ be as in Theorem \ref{formula-corr}.
Assume that $\de$ is a regular ultrafilter on $\lambda$ \emph{(}note that not all of these properties
imply regularity\emph{)}. Then:

$(1) \leftarrow (2) \leftarrow (3) \leftarrow (5) \leftarrow (6)$, with $(1) \not\rightarrow (2)$,  
consistently $(2) \not\rightarrow (3)$,
$(3) \not\rightarrow (5)$, and whether $(5)$ implies $(6)$ is open.
Moreover $(1) \leftarrow (4) \leftarrow (5) \leftarrow (6)$, where $(3) \not\rightarrow (4)$ thus $(2) \not\rightarrow (4)$, $(4) \not\rightarrow (3)$, consistently $(4) \not \rightarrow (5)$, consistently $(4) \not\rightarrow (6)$; and $(4)$ implies $(2)$ is open.
\end{theorem}

\begin{proof}
$(6) \rightarrow (x)$
Since good ultrafilters saturate any countable theory and properties (1)-(5) are all detected by formulas via Theorem \ref{formula-corr},
property (6) implies all the others. 

$(5) \rightarrow (3)$ By Theorem \ref{theory-corr} and the fact that the random graph is minimum among unstable theories. 

$(3) \rightarrow (2)$ By Theorem \ref{formula-corr} lines (2)-(3), i.e. Shelah \cite{Sh:c}.VI.4.8.

$(2) \rightarrow (1)$ Clearly the failure of (1) implies the failure of (2). 

$(1) \not\rightarrow (2)$ Shelah \cite{Sh:c}.VI.5, see Theorem \ref{two}, \S \ref{s:ex2} above.

$(2) \not\rightarrow (3)$ Consistently (assuming an $\aleph_1$-complete ultrafilter) by Theorem \ref{flexible-not-good-2} below. 

$(5) \rightarrow (4)$ Malliaris \cite{mm4} \S 6. 

$(4) \rightarrow (1)$ Proved in \cite{mm-sh-v2}. 

$(3) \not\rightarrow (4)$ Proved in a paper of the authors on excellent ultrafilters \cite{mm-sh3}.

$(4) \not\rightarrow (3)$ Consistently (assuming an $\aleph_1$-complete ultrafilter) by Theorem \ref{flexible-not-good-2} below.

$(4) \not \rightarrow (5), (6)$ We give several proofs of independent interest (each assuming the existence of an $\aleph_1$-complete ultrafilter): 
Theorem \ref{flexible-not-good-2} proves $(4) \not \rightarrow (3)$ thus a fortiori $(4) \not \rightarrow (5)$,
and in \cite{mm-sh-v2}, we give a different proof that $(4) \not \rightarrow (6)$. See also Conclusion \ref{flex-not-good-tp2}.
\end{proof}

At this point, the constructions begin. We remind the reader of the table of contents on p. \pageref{toc}, 
the definitions in \S \ref{s:background}, and the overview of results in \S\ref{s:results}. 

\label{proofs}
\section{$M^\lambda/\de$ is not $(2^\lambda)^+$-saturated for $Th(M)$ unstable} \label{s:trg}

In this section and the next we prove that flexibility does not imply saturation of the random graph,
and thus a fortiori that flexibility does not imply goodness for equality. This gives a proof (assuming the existence
of an $\aleph_1$-compact ultrafilter) that flexible need not mean good.

\begin{fact} \emph{(\cite{Sh:c} Conclusion 1.13 p. 332)} \label{mu-upper-bound}
If $\kappa$ is an infinite cardinal and $\de$ is a regular ultrafilter on $I$ then 
$\kappa^I/\de = \kappa^{|I|}$.
\end{fact}

\begin{claim} \label{random-graph-saturation}
Assume $\lambda \geq \kappa \geq \aleph_0$, $T=T_{rg}$, $M$ a $\lambda^+$-saturated model of $T$, $E$ a uniform ultrafilter
on $\kappa$ such that $|\kappa^\kappa/E| = 2^\kappa$ 
\lp i.e. we can find a sequence $\langle g_\alpha : \alpha < 2^\kappa \rangle$ of members of ${^\kappa \kappa}$
such that $\alpha < \beta \implies g_\alpha \neq g_\beta \mod E$\rp. Then $M^\kappa/E$ is not $(2^{\kappa})^+$-saturated.
\end{claim}

\begin{proof}
Let $\eff = \{ f ~: ~ f: \kappa \times \kappa \rightarrow \{ 0, 1\} \}$, so $|\eff| = 2^\kappa$, and let $\langle f_\alpha : \alpha < 2^\kappa \rangle$
list $\eff$. Let $\langle g_\alpha : \alpha < 2^\kappa \rangle$ be the distinct sequence given by hypothesis. First, for each $\alpha < 2^\kappa$,
we define $\trv_\alpha \in \{0,1\}$ by:
\[ \trv_\alpha = 1 \iff \{ i : f_\alpha(i, g_\alpha(i)) = 1 \} \notin E     \]
Second, since $|M| \geq \kappa$, we may fix some distinguished sequence $\langle a_i : i <\kappa \rangle$ of elements of $M$.
Let $\hat{g}_\alpha \in {^\kappa M}$ be give by $\hat{g}_\alpha(i) = a_{g_\alpha(i)}$. Together these give a set
\[ p(x) = \{ (xR{\hat{g}_\alpha/E})^{\operatorname{if} \trv_\alpha} ~: ~\alpha < 2^{\kappa} \} \]
We check that $p(x)$ is a consistent partial type in $M^\kappa/E$. Since each $\hat{g}_\alpha/E \in M^\kappa/E$, $p$ is a set of formulas
in $M^\kappa/E$. Since $\alpha < \beta \implies g_\alpha/E \neq g_\beta/E \implies \hat{g}_\alpha/E \neq \hat{g}_\beta/E$, the parameters
are distinct and so the type is consistent (note that for the given sequence of parameters, any choice of exponent sequence $\langle \trv_\alpha : \alpha < 2^\kappa \rangle$ would produce a consistent partial type). Moreover, $|p| = 2^\kappa$ again by the choice of $\langle g_\alpha : \alpha < 2^\kappa \rangle$.

We now show that $p(x)$ is omitted in $M^\kappa/E$. Towards a contradiction, suppose that $h \in {^\kappa M}$ were such that $h/E$ realized $p$. 
Let $f: \kappa \times \kappa \rightarrow \{ 0, 1 \}$ be defined by $f(i,j) = 1 \iff M \models h(i) R a_j$. Then $f \in \eff$, hence for some
$\alpha_* < 2^\kappa$ we have that $f_{\alpha_*} = f$. Thus:
\[ \begin{array}{llll}
\trv_{\alpha_*} = 1 &
\oni & M^\kappa/E \models (h/E) R (\hat{g}_{\alpha_*}/E) & \mbox{(by choice of $p$, since $h/E \models p$)} \\ 
& \oni & \{ i < \kappa : M\models h(i) R {a_{g_{\alpha_*}(i)}} \} \in E & \mbox{(by \lost theorem)} \\
& \oni & \{ i : f(i, g_{\alpha_*}(i)) = 1 \} \in E & \mbox{(by the choice of $f$)} \\
& \oni & \{ i : f_{\alpha_*}(i, g_{\alpha_*}(i)) = 1 \} \in E & \mbox{(as $f_{\alpha_*} = f$)} \\
\end{array}
\]
But by definition of the truth values $\trv$, 
\[ \trv_{\alpha_*} = 1 \iff \{ i : f_{\alpha_*}(i, g_\alpha(i)) = 1 \} \notin E     \]
This contradiction completes the proof.
\end{proof}

\begin{cor}
If $\de$ is a $\lambda$-regular ultrafilter on $\lambda$, then $M^\lambda/\de$ is not $(2^\lambda)^+$-saturated
for $M$ a model of any unstable theory. 
\end{cor}

\begin{proof}
Without loss of generality, $M$ is $\lambda^+$-saturated so of cardinality $\geq \kappa$.
The result follows by Claim \ref{random-graph-saturation} and the fact that the random graph is minimum among the unstable theories in Keisler's order.
(That is: by the Claim and Theorem \ref{formula-corr}, $\de$ is not $(2^\kappa)^+$-good for formulas with the independence property,
thus not $(2^\kappa)^+$-good, thus also not $(2^\kappa)^+$-good for formulas with the strict order property.) 
\end{proof}

\begin{obs}
The hypothesis of Claim \ref{random-graph-saturation} is satisfied when $E$ is a regular ultrafilter on $\kappa$ and when
$E$ is a $\kappa$-complete ultrafilter on $\kappa$.
\end{obs}

That is, we want to show that we can find a sequence
$\langle g_\alpha : \alpha < 2^\kappa \rangle$ of members of ${^\kappa \kappa}$ 
so that $\alpha < \beta \implies g_\alpha \neq g_\beta \mod E$.
When $E$ is regular, this follows from Fact \ref{mu-upper-bound}. We give two proofs for the complete case.

\begin{proof}[First proof -- counting functions]
Suppose then that $\kappa$ is measurable, thus inaccessible.
For each $\alpha < \kappa$, let $\Gamma_\alpha = \langle \gamma_\eta : \eta \in {^\alpha 2} \rangle$ be a sequence of pairwise distinct
ordinals $< \kappa$. For each $\eta \in {^\kappa 2}$ let $g_\eta : \kappa \rightarrow \kappa$ be given by 
$g_\eta(\alpha) = \gamma_{\eta|_\alpha}$. So $\{ g_\eta : \eta \in {^\kappa 2} \} \subseteq {^\kappa \kappa}$.
By construction, all we need is one point of difference to know the functions diverge: 
$\eta \neq \nu \in {^\kappa 2}, \eta(\beta) \neq \nu(\beta) \implies 
\{ \alpha < \kappa : g_\eta(\alpha) = g_\nu(\alpha) \} \subseteq \{ \alpha : \alpha < \beta \} = \emptyset \mod E$
as $E$ is uniform.
\end{proof}

\begin{proof}[Second proof -- realizing types]
Suppose then that $\kappa$ is measurable, thus inaccessible.
So we may choose $M$, $|M| = \kappa$ to be a $\kappa$-saturated model of the theory of the random graph. 
To show that $|M^\kappa/E| \geq 2^\kappa$,
it will suffice to show that $2^\kappa$-many distinct types over the diagonal embedding of $M$ in the ultrapower $N$ are realized. 
Let $p(x) = \{ xRf^0_\alpha \land \neg xRf^1_\alpha : \alpha < \kappa \}$ be such a type, with each $f^i_\alpha = {^\kappa\{m\}}$ for some $m \in M$
and of course $\alpha, \beta < \kappa \implies f^0_\alpha \neq f^1_\beta$.
For each $t \in \kappa$, let $p_t(x) = \{ xRf^0_\alpha(t) \land \neg xRf^1_\alpha(t) : \alpha < t \}$. Note that since the elements
$f^i_\alpha$ are constant, for each $t\in\kappa$ we have that $p_t(x)$ is a consistent partial type in $M$.
Choose a new element $h \in {^\kappa M}$ so that
$t \in \kappa$ implies $h(t)$ satisfies $p_t(x)$ in $M$. 
By the saturation of $M$, some such $h$ exists. By uniformity of $E$ $h$ realizes the type $p(x)$, that is, for $\alpha < \kappa$, 
\[ | \kappa \setminus \{ t \in \kappa :  M \models h(t) R f^0_\alpha(t) \land \neg h(t) R f^1_\alpha(t) \} | \leq \alpha < \kappa \]
As no such $h$ can realize two distinct types over $M$ in $N$, we finish.
\end{proof}

\section{For $\kappa$ measurable, $\lambda \geq 2^\kappa$ there is $\de$ on $\lambda$ flexible but not good for $T_{rg}$} \label{s:not-flex2}

We begin by characterizing when flexibility is preserved under products of ultrafilters, Definition \ref{d:prod}.
The first observation says that $\lambda$-flexibility of the first ultrafilter ensures there are $\lambda$-regularizing families in $\de$ below 
certain nonstandard integers, namely those which are a.e. $\de_1$-nonstandard. 

\begin{obs} \label{reg-sets}
Let $\lambda, \kappa \geq \aleph_0$ and let $\de_1$, $E$ be ultrafilters on $\lambda, \kappa$ respectively. Let $\de = \de_1 \times E$
be the product ultrafilter on $\lambda \times \kappa$. Suppose that we are given $n_* \in {^{\lambda \times \kappa}\mathbb{N}}$ such that:

\begin{enumerate}
\item $n \in \mathbb{N} \implies \{ (s,t) \in \lambda \times \kappa ~:~ n_*(s,t) > n \} \in \de$, i.e. $n_*$ is $\de$-nonstandard
\item $N := \{ t \in \kappa: n \in \mathbb{N} \implies \{ s \in \lambda : n_*(s,t) > n \} \in \de_1 \} \in E$
\newline i.e. $E$-almost all of its projections are $\de_1$-nonstandard
\end{enumerate}

Then $(a) \implies (b)$, where:
\begin{enumerate}
\item[(a)] $\de_1$ is $\lambda$-flexible
\item[(b)] there is a regularizing set $\langle X_i : i < \lambda \rangle \subseteq \de$ below $n_*$, 
\newline i.e. such that for all $(s,t) \in \lambda \times \kappa$, $|\{ i < \lambda : (s,t) \in X_i \}| \leq n_*(s,t)$
\end{enumerate}
\end{obs}

\begin{proof}
For each $t \in N$, let $\langle X^t_i : i < \lambda \rangle \subseteq \de_1$ be a regularizing family below $n_*(-,t)$, that is,
such that for each $s \in \lambda$, $| \{ i < \lambda : s \in X^t_i \} | \leq n_*(s,t)$. Such a family is guaranteed by the
$\lambda$-flexibility of $\de_1$ along with the definition of $N$, since the latter ensures that $n_*(-,t) \in {^\lambda \lambda}$ 
is $\de_1$-nonstandard. Now define $\langle X_i : i < \lambda \rangle$ by $X_i = \{ (s,t) : s \in X^t_i \}$. We verify that:

\begin{itemize}
\item $\langle X_i : i <\lambda \rangle \subseteq \de$, as 
$\{ t \in \kappa : \{ s \in \lambda : (s,t) \in X_i \} \in \de_1 \} \supseteq N$ and $N \in E$ by hypothesis.
\item $\langle X_i : i < \lambda \rangle$ is below $n_*$, since for each $(s,t) \in \lambda \times \kappa$,
\\ $| \{ i : (s,t) \in X_i \}| = | \{ i : (s,t) \in X^t_i \} | \leq n_*(s,t)$ by construction. 
\end{itemize}

This completes the proof.
\end{proof}

The next claim shows that $\de$-nonstandard integers project $E$-a.e. to $\de_1$-nonstandard integers 
precisely when the second ultrafilter $E$ is at least $\aleph_1$-complete. 
 
\begin{claim} \label{initial-segment}
Let $\lambda, \kappa \geq \aleph_0$ and let $\de_1$, $E$ be uniform ultrafilters on $\lambda, \kappa$ respectively. Let $\de = \de_1 \times E$
be the product ultrafilter on $\lambda \times \kappa$.  Then the following are equivalent.

\begin{enumerate}
\item If $n_* \in {^{\lambda \times \kappa}\mathbb{N}}$ is such that 
$n \in \mathbb{N} \implies \{ (s,t) \in \lambda \times \kappa ~:~ n_*(s,t) > n \} \in \de$,   
\\ then $N \in E$ where $N := \{ t \in \kappa: n \in \mathbb{N} \implies \{ s \in \lambda : n_*(s,t) > n \} \in \de_1 \}$.
\item $E$ is $\aleph_1$-complete.
\end{enumerate}
\end{claim}

\begin{proof}
(1) $\rightarrow$ (2) Suppose $E$ is not $\aleph_1$-complete, so it is countably incomplete and we can find $\langle X_n : n < \omega \rangle \subseteq E$
such that $\bigcap \{ X_n : n \in \omega \rangle \} = \emptyset \mod \de$. Without loss of generality, $n < \omega \rightarrow X_{n+1} \subsetneq X_n$. 
Let $n_* \in {^{\lambda \times \kappa} \mathbb{N}}$ be given by:
\[ t \in \kappa \land t \in X_n \setminus X_{n+1} \implies n_*(-,t) = {^\lambda\{n\}} \]
Then $n_*$ is $\de$-nonstandard but its associated set $N$ is empty (as a subset of $\kappa$, so a fortiori empty modulo $\de$). 

(2) $\rightarrow$ (1) Suppose on the other hand that $E$ is $\aleph_1$-complete, and let some $\de$-nonstandard $n_*$ be given.
For each $n \in \mathbb{N}$, define $X_n = \{ t \in \kappa : \{ s \in \lambda : n_*(s,t) > n \} \in \de_1 \}$. Then by completeness, 
$N \supseteq \bigcap\{ X_n : n \in \mathbb{N} \} \in E$. 
\end{proof}

\begin{cor} \label{reg-transfer}
Let $\lambda, \kappa \geq \aleph_0$ and let $\de_1$, $E$ be ultrafilters on $\lambda, \kappa$ respectively where $\kappa > \aleph_0$ 
is measurable. Let $\de = \de_1 \times E$ be the product ultrafilter on $\lambda \times \kappa$. Then:
\begin{enumerate}
\item If $\de_1$ is $\lambda$-flexible and $E$ is $\aleph_1$-complete, then $\de$ is $\lambda$-flexible.
\item If $\lambda \geq \kappa$ and $\lcf(\aleph_0, \de_1) \geq \lambda^+$, then $\lcf(\aleph_0, \de) \geq \lambda^+$, so in particular,
$\de = D \times E$ will $\lambda^+$-saturate any countable stable theory.
\end{enumerate}
\end{cor}

\begin{proof}
(1) By Claim \ref{initial-segment} and Observation \ref{reg-sets}.

(2) Let us show that the $\de_1$-nonstandard integers are cofinal in the $\de$-nonstandard integers. 
Let $M = (\mathbb{N}, <)^\lambda/\de_1$, $N = M^\kappa/E$. Let $n_* \in N$ be $\de$-nonstandard.
By Claim \ref{initial-segment}, the set
$X = \{ t \in \kappa : n_*(t)~\mbox{is a $\de_1$-nonstandard element of $M$} \} \in E$. 
Since $\lcf(\aleph_0, \de_1) \geq \lambda^+ > \kappa$, there is $m_* \in M$ which is $\de_1$-nonstandard and
such that $t \in X \implies M \models n_*(t) > m_*$. Then the diagonal embedding of $m_*$ 
in $N$ will be $\de$-nonstandard but below $n_*$, as desired. 
The statement about stable theories follows by \S \ref{s:cont-imp}, Theorem \ref{theory-corr} and Theorem \ref{implications}(2)$\rightarrow$(1).
\end{proof}

\begin{theorem} \label{flexible-not-good-2}
Assume $\aleph_0 < \kappa < \lambda = \lambda^\kappa$, $2^\kappa \leq \lambda$, $\kappa$ measurable. 
Then there exists a regular uniform ultrafilter $\de$ on $\lambda$ such that $\de$ is $\lambda$-flexible, yet
for any model $M$ of the theory of the random graph, $M^\lambda/\de$ is not $(2^\kappa)^+$-saturated. However,
$\de$ will $\lambda^+$-saturate any countable stable theory.

A fortiori, $\de$ is neither good nor good for equality, and it will fail to $(2^\kappa)^+$-saturate any unstable theory.
\end{theorem}

\begin{proof}
Let $\ee$ be a uniform $\aleph_1$-complete ultrafilter on $\kappa$. 
Let $\de_1$ be any $\lambda$-flexible (thus, $\lambda$-regular) ultrafilter on $\lambda$, e.g. 
a regular $\lambda^+$-good ultrafilter on $\lambda$. Let $\de = \de_1 \times \ee$ be the product ultrafilter
on $I = \lambda \times \kappa$. Then $|I| = \lambda$, and we have that $\de$ is $\lambda$-flexible by 
Corollary \ref{reg-transfer}(1), it saturates countable stable theories by \ref{reg-transfer}(2),
and it fails to $(2^\kappa)^+$-saturate the random graph by Claim \ref{random-graph-saturation}. 

The last line of the Theorem follows from the fact that the random graph is minimum among unstable theories in Keisler's order,
along with the fact that both goodness and goodness for equality are necessary conditions on regular ultrafilters for saturating
some unstable theory (the theory of dense linear orders and the theory $T^*_{feq}$ of infinitely many independent equivalence relations,
respectively).  
\end{proof}

\begin{cor} \label{lcf-rg}
In the construction just given, by Claim \ref{initial-segment}, $\lcf(\aleph_0, \de) = \lcf(\aleph_0, \de_1) \geq \lambda^+$ since $\de_1$
is $\lambda^+$-good and the nonstandard integers of $\de_1$ are cofinal in the nonstandard integers of $\de$. 
Thus consistently, a regular ultrafilter on $\lambda > \kappa$ may have large lower cofinality of $\aleph_0$ while
failing to $(2^\kappa)^+$-saturate the random graph. 
\end{cor}

By \cite{Sh:c}.VI.4, a necessary condition for a regular ultrafilter $\de$ on $\lambda$ 
to saturate some unstable theory is that $\lcf(\aleph_0, \de) \geq \lambda^+$; Corollary \ref{lcf-rg} shows it is not sufficient. 

\section{$M^\lambda/\de$ is not $\lambda^{++}$-saturated for $\lambda$ regular and $Th(M)$ non-simple} \label{s:non-simple}

In this section we prove that there is a loss of saturation in ultrapowers of non-simple theories. As mentioned above, this is a new proof 
of a result from \cite{Sh:c}.VI.4.7, which reflects an interest (visible elsewhere in this paper e.g. \ref{kappa-plus-cut}) in controlling the distribution of sets of indices. 

\begin{defn}
A first-order theory has $TP_2$ if there is a formula $\vp(x;\overline{y})$ which does, where this means that in any $\aleph_1$-saturated model $M \models T$
there exists an array $A = \{ \overline{a}^i_j : i <\omega, j<\omega \}$ of tuples, $\ell(\overline{a}^i_j) = \ell(\overline{y})$ such that: 
for any finite $X \subseteq \omega \times \omega$, the partial type
\[ \{ \vp(x;\overline{a}^i_j): (i,j) \in X \} \] 
is consistent if and only if
\[ (i,j) \in X \land (i^\prime, j^\prime) \in X \land i=i^\prime \implies j=j^\prime \]
Informally speaking, by choosing no more than one tuple of parameters
from each column we form a consistent partial type. 
\end{defn}

\begin{defn} \label{sop2-tree}
A first-order theory has $TP_1$, or equivalently $SOP_2$, if there is a formula $\vp(x;\overline{y})$ which does, where this means that in any $\aleph_1$-saturated model $M \models T$
there exist $\langle \overline{a}_\eta : \eta \in {^{\aleph_0 > } 2 } \rangle$ such that: 

\begin{enumerate}
\item for $\eta, \rho \in {^{\aleph_0 > } 2}$ incomparable, i.e. $\neg (\eta \tlf \rho) \land \neg (\rho \tlf \eta)$, we have that
$\{ \vp(x;\overline{a}_\eta), \vp(x;\overline{a}_\rho) \}$ is inconsistent.
\item for $\eta \in {^{\aleph_0 } 2}$, $\{ \vp(x;\overline{a}_{\eta|_i}) : i < \aleph_0 \}$ is a consistent partial $\vp$-type.
\end{enumerate}
\end{defn}

\begin{fact} \label{tp1-tp2} \emph{(See \cite{Sh:c}.III.7, \cite{Sh93})}
If $T$ is not simple then $T$ contains either a formula with $TP_1$ (equivalently $SOP_2$) or a formula with $TP_2$.
\end{fact}

\begin{fact} \label{tree-claim}
Suppose that $\de$ is a regular uniform ultrafilter on $\lambda$, $\lambda = \lambda^{<\lambda}$ or just
$(\aleph_1, \aleph_0) \rightarrow (\lambda^+, \lambda)$ \lp see \cite{KeSh:769}\rp. 
Let $\kappa = \aleph_0$. Then for each $\epsilon < \lambda$ 
we may choose a sequence of sets $\overline{u}_\epsilon = \langle u_{\epsilon, \alpha} : \alpha < \lambda^+ \rangle$ such that:
\begin{enumerate}
\item $u_{\epsilon, \alpha} \subseteq \alpha$
\item $ | u_{\epsilon, \alpha} | < \lambda $
\item $\alpha \in u_{\epsilon, \beta} \implies u_{\epsilon, \alpha} = u_{\epsilon, \beta} \cap \alpha$
\item if $ u \subseteq \lambda^+$, $|u| < \kappa$ then 
\[ \{ \epsilon < \lambda : \exists \alpha (u \subseteq u_{\epsilon, \alpha}) \} \in \de \]
\end{enumerate}
\end{fact}

\begin{proof}
By Kennedy-Shelah-Vaananen \cite{KSV} p. 3 this is true when $\lambda$ satisfies the stated hypothesis and $\de$ is regular. 
Note that as briefly mentioned there, in the case of singular $\lambda$, the claim may be true; but it is also consistent that it may fail. 
\end{proof}

We will use $\kleq$ to indicate comparability in the $TP_1$ tree, i.e. $\eta \kleq \rho$ means $\eta$ is before $\rho$ in the partial tree order. 

\begin{claim} \label{tp1-claim}
Given $\lambda \geq \aleph_0$ regular, let $\de$ be a regular uniform ultrafilter on $\lambda$.
Suppose $T$ has $TP_1$, as witnessed by $\vp$, and let $M \models T$ be $\lambda^{++}$-saturated. Then $M^\lambda/\de$ is not $\lambda^{++}$-saturated,
and in particular is not $\lambda^{++}$-saturated for $\vp$-types.
\end{claim}

\begin{proof}
Fix $\vp=\vp(x;\overline{y})$ a formula with $TP_1$.
By the hypothesis of $TP_1$ and saturation of $M$, we may choose parameters $\overline{a}_\eta \in M$ ($\ell(a_\eta) = \ell(y)$) 
for $\eta \in^{^{\lambda^+ > } \lambda}$ such that:
\begin{itemize}
\item $\{ \vp(x; a_{\eta|_i}) : i < \lambda^+ \}$ is a consistent partial type for each $\eta \in^{^{\lambda^+ > } \lambda}$
\item if $\eta, \nu \in^{^{\lambda^+ > } \lambda}$ are $\kleq$-incomparable then the set $\{ \vp(x; a_\eta), \vp(x;a_\nu) \}$
is contradictory.
\end{itemize}

Let $\langle \overline{u_\epsilon} : \epsilon < \lambda \rangle$ be as given by Fact \ref{tree-claim}. For each $\epsilon < \lambda$,
we choose by induction on $\alpha < \lambda^+$ indices $\langle \eta_{\epsilon, \alpha} : \alpha < \lambda^+ \rangle$ such that:
\begin{enumerate}
\item[(i)] $\eta_{\epsilon, \alpha} \in {^\alpha \lambda}$
\item[(ii)] $\eta_{\epsilon, \alpha} \triangleleft \eta_{\epsilon, \beta}$ iff $\alpha \in u_{\epsilon, \beta}$
\end{enumerate}

Informally, we choose indices for nodes of the tree so that consistency of the associated formulas reflects the
structure of the sets $\langle \overline{u_\epsilon} : \epsilon < \lambda \rangle$, which we can do by
the assumption on saturation of $M$ and the downward coherence condition, Fact \ref{tree-claim}(3). 

For each $\alpha < \lambda^+$ we thus have an element $f_\alpha \in M^\lambda$ given by $f_\alpha(\epsilon) = a_{\eta_{\epsilon, \alpha}} \in M$.
The sequence $\langle f_\alpha/\de : \alpha < \lambda^+ \rangle$ is a sequence of $\lambda^+$ members of $M^\lambda/\de$. Moreover, by the downward
coherence condition, if $n<\omega$, $\alpha_0 < \dots < \alpha_n < \lambda^+$ and for some $\alpha$, $\alpha_n < \alpha < \lambda$ 
we have that $\{ \alpha_0, \dots \alpha_n \} \subseteq \alpha$, then in fact $ \ell < n \implies \alpha_\ell \in u_{\alpha_{\ell + 1 }}$; 
note here that the identity of $\alpha$ is not important, only its existence. 
Thus for any $n < \omega$ and any $\alpha_0 < \dots < \alpha_n < \lambda^+$, Fact \ref{tree-claim}(4) implies that
\[ \{ \epsilon < \lambda : \eta_{\epsilon, \alpha_0} \triangleleft \eta_{\epsilon, \alpha_1} 
\triangleleft \dots \triangleleft \eta_{\epsilon, \alpha_n} \} \in \de \]
and therefore 
\[ \{ \epsilon < \lambda : M \models \exists x \bigwedge_{\ell} \vp(x; a_{\epsilon, \alpha_\ell} ) \} \in \de \]
Since $n$, $\alpha_0, \dots \alpha_n$ were arbitrary, we have verified that 
$p = \{ \vp(x; f_\alpha/\de) : \alpha < \lambda^+ \}$ is a consistent partial type.

Assume towards a contradiction that $p$ is realized, say by $f \in {^\lambda M}$. For each $\alpha < \lambda^+$, define
$J_\alpha = \{ \epsilon < \lambda : M \models \vp(f(\epsilon), f_\alpha(\epsilon)) \}$ and note that $J_\alpha \in \de$
since we assumed $f$ realizes the type. By definition of $f_\alpha$, if $\epsilon \in J_\alpha$ then 
$M \models \vp(f(\epsilon), a_{\eta_{\epsilon, \alpha}})$. For each $\alpha < \lambda^+$, as $J_\alpha \neq \emptyset$
we may choose some $\epsilon_\alpha \in J_\alpha$. Since $\lambda^+ > \lambda$ is regular, there is some 
$\epsilon_* < \lambda$ such that $S = \{ \alpha < \lambda^+ : \epsilon_\alpha = \epsilon_* \} \subseteq \lambda^+$
is unbounded in $\lambda^+$. By definition of $J_\alpha$, this means that $p_* = \{ \vp(x,a_{\eta_{\epsilon_*, \alpha}}) : \alpha \in S \}$
is a consistent partial type in $M$. By definition of $TP_1$, if $p_*$ is a consistent partial type it must be that
$X = \{ \eta_{\epsilon_*, \alpha} : \alpha \in S \}$ has no two $\kleq$-incomparable elements. 
Since $|S| = \lambda^+$, we may choose some $\alpha \in S$ such that $|S \cap \alpha| = \lambda$. 
But this contradicts the choice of indices $\eta_{\epsilon, \alpha}$ in (i)-(ii) above, in light of Fact \ref{tree-claim}(2). 
\end{proof}

We now consider $TP_2$.

\begin{obs} \label{array-claim}
Let $|I| = \lambda \geq \aleph_0$ and let $\de$ be a regular ultrafilter on $\lambda$. There are functions
$\langle \nu_\alpha : \lambda \leq \alpha < \lambda^+ \rangle$ such that:
\begin{itemize}
\item $\nu_\alpha : I \rightarrow \alpha$
\item for any $\alpha < \beta < \lambda^+$, 
$ \{ i \in I : \nu_\alpha(i) = \nu_\beta(i) \}  \notin \de$
\end{itemize}
\end{obs}

\begin{proof}
This follows e.g. from the statement that for any regular ultrafilter $\de$ on $\lambda$,
$|\mathbb{N}^\lambda/\de| = 2^\lambda$. 
\end{proof}

\begin{claim} \label{tp2-claim}
Given $\lambda \geq \aleph_0$, let $\de$ be a regular uniform ultrafilter on $I$, $|I| = \lambda$.
Suppose $T$ has $TP_2$, as witnessed by $\vp$, and let $M \models T$ be $\lambda^{++}$-saturated. Then $M^\lambda/\de$ is not $\lambda^{++}$-saturated,
and in particular is not $\lambda^{++}$-saturated for $\vp$-types.
\end{claim}

\begin{proof}
Let $\langle \nu_\alpha : \alpha < \lambda^+ \rangle$ be as given by Observation \ref{array-claim}. By the definition of $TP_2$, for some formula 
$\vp = \vp(x;\overline{y})$ we have an array $\{ {a}_{\alpha, \beta} : \alpha, \beta < \lambda^+ \} \subseteq M$, with
$\ell({a}) = \ell(\overline{y})$, such that:
\begin{enumerate}
\item for $\beta < \gamma < \lambda^+$, $\{ \vp(x;a_{\alpha, \beta}), \vp(x; a_{\alpha, \gamma}) \}$ is contradictory
\item for any $n < \omega$ and $\alpha_0 < \dots < \alpha_n < \lambda^+$, and any $\{ \beta_0, \dots \beta_n \} \subseteq \lambda^+$,
the set $\{ \vp(x;a_{\alpha_i, \beta_i}) : i \leq n \}$ is a consistent partial type
\end{enumerate}
Informally, the columns consist of parameters for pairwise contradictory instances of $\vp$, whereas choosing any sequence of parameters
with no more than one in any given column produces a consistent partial type. 

Now we define for each $\alpha$, $\lambda \leq \alpha < \lambda^+$ functions $f_\alpha \in M^I$ by: $f_\alpha(i) = a_{\nu_\alpha(i), \alpha}$. 
Let us verify that $p(x) = \{ \vp(x; f_\alpha/\de) : \alpha < \lambda^+ \}$ is a consistent partial type. For any $n < \omega$ and
$\alpha_0 < \dots < \alpha_n < \lambda^+$, we have by the second item in Observation \ref{array-claim} that 
$\{ i \in I : \bigwedge\{ v_{\alpha_j}(i) \neq v_{\alpha_k}(i) : j < k \leq n \} \in \de$, and thus by definition of the $TP_2$ array,
$\{ \vp(x;f_{\alpha_j}/\de : j \leq n \}$ is consistent. 

Suppose for a contradiction that $p$ were realized in $M^\lambda/\de$, say by $f \in {^IM}$. We proceed analogously to Claim \ref{tp1-claim}.
For each $\alpha < \lambda^+$, define
$J_\alpha = \{ \epsilon < \lambda : M \models \vp(f(\epsilon), f_\alpha(\epsilon)) \}$, which again will be in $\de$ (and in particular, nonempty)
since we assumed $f$ realizes the type. For each $\alpha < \lambda^+$, choose some $\epsilon_\alpha \in J_\alpha$, and let 
$S \subseteq \lambda^+$ be a stationary set on which this choice is constant, call it $\epsilon_*$. Now we look at
$p_* = \{ \vp(x,a_{\nu_\alpha(\epsilon_*), \alpha}) : \alpha \in S \}$. By definition of $S$, $f(\epsilon_*)$ realizes this type,
so in particular it is consistent.  But since $S$ is stationary and $\nu_\alpha(\epsilon) < \alpha$,
by Fodor's lemma there are $\alpha \neq \beta$ from $S$ such that $\nu_\alpha (\epsilon_*) = \nu_\beta (\epsilon_*) =: \gamma$. But then
$\{ \vp(x, a_{\gamma, \alpha}), \vp(x, a_{\gamma, \beta}) \} \subseteq p_*$ is contradictory, and this completes the proof.
\end{proof}

\begin{concl} \label{simple-non-sat} 
Let $\lambda = \lambda^{<\lambda}$ and let $\de$ be a regular ultrafilter on $\lambda$. If $T$ is not simple and $M \models T$, then 
there is $\vp$ such that $M^\lambda/\de$ is not $\lambda^{++}$-saturated for $\vp$-types.
\end{concl}

\begin{proof}
By Claim \ref{tp1-claim}, Claim \ref{tp2-claim} and Fact \ref{tp1-tp2}.
\end{proof}

\begin{rmk}
Let $D_1, D_2$ be ultrafilters on $\lambda, \kappa$ respectively and suppose that $\kappa = \kappa^{<\kappa}$.
If $\lambda \geq \kappa^+$ and $D_2$ is regular, then by Theorem \ref{formula-corr}(5) and Conclusion 
\ref{simple-non-sat}, $D_1 \times D_2$ cannot be good for equality.  
\end{rmk}

\section{$\kappa$-complete not $\kappa^+$-complete implies no $(\kappa, \kappa)$-cuts} \label{s:kappa-cuts}

\begin{claim} \label{complete-cuts1}
Suppose that $\ee$ is a $\kappa$-complete but not $\kappa^+$-complete ultrafilter on $I$ and
$M_1$ is a $\kappa^+$-saturated model in which a linear order $L$ and tree $T$ are interpreted. 
Then in $M_2 = M_1^I/\ee$ :
\begin{enumerate}
\item[(a)] the linear order $L^{M_2}$ has no $(\kappa, \kappa)$-cut, and moreover no $(\theta, \sigma)$-cut
for $\theta, \sigma < \kappa$ both regular.
\item[(b)] the tree $T^{M_2}$ has no branch (i.e. maximal linearly ordered set) of cofinality $\leq \kappa$.
\end{enumerate}
\end{claim}

\begin{rmk} In the statement of Claim \ref{complete-cuts1}:
\begin{enumerate}
\item in \emph{(a)}, the $\kappa$-saturation of $M_1$ is necessary in the following sense: if there is a sequence
$\overline{\theta} = \langle \theta_t : t \in I \rangle$, which certainly need not be distinct, 
such that $M$ has a $(\theta_t, \theta_t)$-cut for each 
$t \in I$ and $(\prod_{t \in I} \theta_t, <_\ee)$ has cofinality $\kappa$, then the conclusion of Claim \ref{complete-cuts1}(a) is false. 
\item By this Claim, we may add to the conclusion of
Theorem \ref{flexible-not-good-2} that $(\kappa, \kappa) \notin \mc(\ee)$, Definition \ref{mc}, since in that Theorem 
the ultrafilter $E$ is a $\kappa$-complete uniform ultrafilter on $\kappa$ and thus not $\kappa^+$-complete. 
\end{enumerate}
\end{rmk}

\begin{proof} (of Claim \ref{complete-cuts1})

(a)
The ``moreover'' clause in (a) follows from the fact that $M_1$ and $M_2$ are $L_{\kappa, \kappa}$-equivalent, by the completeness of $\ee$,
and the hypothesis on saturation of $M_1$.

So we consider a potential $(\kappa, \kappa)$-cut in $M_2$, i.e. a $(\kappa, \kappa)$-pre-cut 
given by $\langle f_\alpha : \alpha < \kappa \rangle$, $\langle g_\alpha : \alpha < \kappa \rangle$
where if $\alpha < \beta < \kappa$ then 
\[ M_2 \models (f_\alpha/\ee) <_L (f_\beta/\ee) <_L < (g_\beta/\ee) <_L (g_\alpha/\ee) \]

For $0 < \gamma < \kappa$ let
\[ A_\gamma = \{ t : ~\mbox{if}~ \alpha < \beta < \gamma ~\mbox{then} ~M_1 \models f_\alpha(t) <_L f_\beta(t) <_L g_\beta(t) <_L g_\alpha(t) \} \]

Let $A_0 = A_1 = I$. Then $\overline{A} = \langle A_\gamma : \gamma < \kappa \rangle$ is a continuously decreasing sequence of elements of $\ee$, i.e.:
\begin{itemize}
\item $\gamma_1 < \gamma_2 \implies A_{\gamma_1} \supseteq A_{\gamma_2}$
\item for limit $\delta < \kappa$, $A_\delta = \bigcap \{ A_\gamma : \gamma < \delta \}$
\item each $A_\gamma \in \ee$, by choice of the functions and $\kappa$-completeness
\end{itemize}

As we assumed $\ee$ is $\kappa$-complete but not $\kappa^+$-complete, there is a sequence $\overline{B} = \langle B_\gamma : \gamma < \kappa \rangle$
of elements of $\ee$ such that $\bigcap\{ B_\gamma : \gamma < \kappa \} = \emptyset$.
We may furthermore assume that $\overline{B}$ is a continuously decreasing sequence (if necessary, inductively replace $B_\delta$
by $\bigcap \{ B_\gamma : \gamma < \delta \}$ using $\kappa$-completeness).

Thus given $\overline{A}, \overline{B}$, for each $t \in I$ we may define 
\[ \gamma(t) = \operatorname{min} \{ \alpha : t \notin A_{\alpha+1} \cap B_{\alpha+1} \} \]

By choice of $\overline{B}$, $t \mapsto \gamma(t)$ is a well-defined function from $I$ to $\kappa$, and 
$t \in A_{\gamma(t)} \cap B_{\gamma(t)}$. Recall that we want to show that our given $(\kappa, \kappa)$-sequence is not a cut. 
Choose $f_\kappa, g_\kappa \in {^IM}$ so that first, for each $t \in I$, $f_\kappa(t), g_\kappa(t) \in L^{M_1}$,
and second, for each $t \in I$ and all $\alpha < \gamma(t)$, 
\[ M_1 \models f_\alpha(t) \leq_{L} f_\kappa(t) <_L g_\kappa(t) \leq_L g_\alpha(t) \]
This we can do by the choice of $\overline{A}$ as a continuously decreasing sequence (so the function values $f_\alpha, g_\beta$ 
below $\gamma(t)$ in each index model are correctly ordered) 
and the saturation hypothesis on $M_1$. Thus for each $\alpha < \kappa$, we have that
\[ \{ t : f_\alpha(t) \leq_L  f_\kappa(t) <_L g_\kappa(t) \leq_L g_\alpha(t) \} \supseteq A_{\alpha+1} \cap B_{\alpha+1} \in \ee \]
which completes the proof.

(b) Similar proof, but we only need to use one sequence $\langle f_\alpha : \alpha < \kappa \rangle$ which we choose to 
potentially witness that the cofinality of the branch is at most $\kappa$. 
\end{proof} 

We prove a related fact for normal filters, Definition \ref{d:normal}.

\begin{claim} \label{n19}
Assume $\ee$ is a normal filter on $\lambda$ and $M$ is a $\lambda^+$-saturated dense linear order. Then 
$M^I/\ee$ is $\lambda^+$-saturated.
\end{claim}

\begin{proof}
Suppose that $\langle f_\alpha/\ee : \alpha < \kappa_1 \rangle$ is increasing in $M^I/\ee$, and
$\langle g_\beta/\ee : \beta < \kappa_2 \rangle$ is decreasing in $M^I/\ee$, with $\kappa_1, \kappa_2 \leq \lambda$
and $f_\alpha/\ee < g_\beta/\ee$ for $\alpha<\kappa_1, \beta<\kappa_2$. Let
\[ X_{\alpha, \beta} = \{ t \in \lambda  : f_\alpha(t) < g_\beta(t) \} \in \ee \]
for $\alpha < \beta < \lambda$.    
Without loss of generality, suppose $\kappa_1 \leq \kappa_2$. For each $\beta < \kappa_2$, let
\[ Y_\beta = \{ \alpha \in \lambda : (\forall j < (1+\alpha) \cap \beta) (j \in X_{\alpha, \beta}) \}  \in \ee \]
by normality. For $\beta \geq \kappa_2$, let $Y_\beta = I$. Finally, define
\[ Z = \{ \beta \in \lambda : (\forall k < (1+\beta) \cap \kappa_2 ) ( j \in Y_\beta) \} \in \ee \]
Now if $t \in Z$ (so $t$ plays the role of $\beta$) we have that 
\[ p_t = \{ f_\alpha(t) < x < g_t(t) : \alpha < t \} \]
is a consistent partial type, realized in $M$ by the saturation hypothesis. Choose
$h \in {^\lambda M}$ such that for each $t \in Z$, $h \models p_t$. Then $h$ realizes the type. 
\end{proof}

Note that Claim \ref{n19} implies by Fact \ref{fact-good} of the Appendix that the relevant 
$\ee$ is good.

The use of the additional hypothesis ``normal'' in this section comes in Step 3 of Claim \ref{kappa-plus-cut},
and consequently in later results which rely on it. Recall Definition \ref{d:normal} and Fact \ref{normal-fodor}.

\begin{claim} \label{kappa-plus-cut}
Assume $\kappa$ measurable, $\ee$ a normal $\kappa$-complete ultrafilter on $\kappa$, 
$\lambda \geq \kappa$,
$M_1$ a $\lambda$-saturated model with ${(L_M, <_M)}$ a dense linear order. Let $M_2 = M_1^\kappa/\ee$. Then
$L_{M_2}$ has a $(\kappa^+, \kappa^+)$-cut.
\end{claim}

\begin{proof}
The proof has several steps.

\step{Step 1: Fixing sequences of indices.}
For each $\alpha < \kappa^+$ choose $\overline{\uu}_\alpha = \langle u_{\alpha, \epsilon} : \epsilon < \kappa \rangle$ so that:

\begin{enumerate}
\item[(a)] this sequence is $\subseteq$-increasing and continuous, and for each $\alpha < \kappa^+$, $u_{\alpha, 0} = \emptyset$ 
\item[(b)] for each $\epsilon < \kappa$, $|u_{\alpha, \epsilon}| < \kappa$
\item[(c)] $\bigcup \{ u_{\alpha, \epsilon} : \epsilon < \kappa \} = \alpha$
\item[(d)] (for coherence) for $\beta < \alpha < \kappa^+$,
\[ \beta \in u_{\alpha, \epsilon} \implies u_{\beta, \epsilon} \subseteq u_{\alpha, \epsilon} \]
\end{enumerate}
Such a sequence will always exist as $|\alpha| \leq \kappa$.
[Details: Clearly such a sequence exists for $\alpha \leq \kappa$: let $u_{\kappa, \epsilon} = \epsilon \cap \alpha$, 
so for arbitrary $\kappa \leq \alpha < \kappa^+$,
fixing a bijection to $\kappa$ let $\overline{V}_\alpha = \langle v_{\alpha, \epsilon} : \epsilon < \kappa \rangle$  
be the preimage of the sequence for $\kappa$.
Having thus fixed, for each $\alpha < \kappa^+$, a sequence satisfying (a)-(c) we may then inductively pad these sequences to ensure
coherence. For $\beta = 0$ and each $\epsilon < \kappa$, let $u_{\beta, \epsilon} = v_{\beta, \epsilon}$.
For $0 < \beta < \kappa^+$, for each $\epsilon < \kappa$ let 
$u_{\beta, \epsilon} = \bigcup \{ u_{\alpha, \epsilon} : \alpha \in v_{\beta, \epsilon} \}$, and note that this will preserve
(b), (a), (c) and ensure (d).]

\step{Step 2: The inductive construction of the (pre-)cut.}
We now construct a cut. We will first describe the construction, and then show that it is in fact a cut (i.e. we will show that we have indeed
constructed a pre-cut, and that this pre-cut is not realized).

By induction on $\alpha < \kappa^+$ we will choose $f_\alpha, g_\alpha \in {^\kappa(L_{M_1})}$. The intention is that
each $u_{\alpha, \epsilon}$ represents a small set of prior functions which we take into account when choosing the values for
$f_\alpha, g_\alpha$ at the index $\epsilon \in \kappa$. 

At stage $\alpha$, for each index $\epsilon < \kappa$ define 
\begin{eqnarray*}
 w_{\alpha, \epsilon} = \{ \beta \in u_{\alpha, \epsilon} :&  
 \langle f_\gamma(\epsilon) : \gamma \in u_{\alpha, \epsilon} \cap (\beta + 1) \rangle ~\mbox{is $<_{L(M_1)}$-increasing}, \\
&  \langle g_\gamma(\epsilon) : \gamma \in u_{\alpha, \epsilon} \cap (\beta + 1) \rangle ~\mbox{is $<_{L(M_1)}$-decreasing}, \\
& \mbox{and}~ f_\beta(\epsilon) <_{L(M_1)} g_\beta(\epsilon) ~~\}   \\
\end{eqnarray*}

Our aims in defining $f_\alpha, g_\alpha$ are, on the one hand, to continue describing a pre-cut, and on the other, to stay as close to the
linearly ordered $w_{\alpha, \epsilon}$ as possible, as we now describe. That is, for fixed $\alpha$ for each $\epsilon$, we will
choose $f_\alpha, g_\alpha$ so that:

\begin{enumerate}
\item[(e)] For all $\beta \in w_{\alpha, \epsilon}$,
$M_1 \models f_{\beta}(\epsilon) < f_\alpha (\epsilon) < g_\alpha(\epsilon) < g_{\beta}(\epsilon)$

i.e. \emph{locally we continue the pre-cut described by $w_{\alpha, \epsilon}$.}
\br
\item[(f)] For all $\beta < \alpha$, neither $M_1 \models f_\beta(\epsilon) < f_\alpha(\epsilon) < g_\beta(\epsilon) < g_\alpha(\epsilon)$ nor
$M_1 \models f_\alpha(\epsilon) < f_\beta(\epsilon) < g_\alpha(\epsilon) < g_\beta(\epsilon)$

i.e. \emph{the intervals are either nested or disjoint.} 
\br
\item[(g)] If $\gamma \in \alpha$ and $f_\gamma(\epsilon), g_\gamma(\epsilon)$ satisfy:
\[ \beta \in w_{\alpha, \epsilon} \implies f_\beta(\epsilon) < f_\gamma(\epsilon) < g_\gamma(\epsilon) < g_\beta(\epsilon) \]
then the intervals $[f_\alpha(\epsilon), g_\alpha(\epsilon)]_{L(M_1)}$, 
$[f_\beta(\epsilon), g_\beta(\epsilon)]_{L(M_1)}$ are disjoint

i.e. \emph{inside the pre-cut given by $w_{\alpha, \epsilon}$ we avoid any further refinements: we realize \emph{exactly}
the intitial segment given by $w_{\alpha, \epsilon}$.}
\end{enumerate}

We will show in step 3 that by the hypothesis on $\kappa$, (a)-(g) imply the further condition that for each fixed $\alpha < \kappa^+$, 
\begin{enumerate}
\item[(h)] For all $\beta < \alpha$, $\{ \epsilon : f_{\beta}(\epsilon) < f_\alpha (\epsilon) < g_\alpha(\epsilon) < g_{\beta}(\epsilon) \} \in \ee$, 
i.e. \emph{the functions chosen will ultimately describe a pre-cut.}
\end{enumerate}

At each index $\epsilon < \kappa$, we may choose $f_\alpha(\epsilon), g_\alpha(\epsilon)$ to satisfy (e),(f),(g) simply because
$L_{M_1}$ is dense and $\lambda^+$-saturated; the definition of $w_{\alpha, \epsilon}$ ensures (e) describes a pre-cut; 
and since (f) is inductively satisfied, (g) is possible.

\step{Step 3: For $\beta < \alpha$, $f_\beta < f_\alpha < g_\alpha < g_\beta$.}
In this step we verify that for the objects constructed in the previous step, for each $\alpha < \kappa^+$ 
and all $\beta < \alpha$, 
\[ X_{\alpha, \beta} = \{ \epsilon < \kappa : 
\beta \in u_{\alpha, \epsilon}, f_{\beta}(\epsilon) < f_\alpha (\epsilon) < g_\alpha(\epsilon) < g_{\beta}(\epsilon) \} \in \ee \]

(By conditions (a)-(d) requiring $\beta \in u_{\alpha, \epsilon}$ does not affect membership in $\ee$.)
Suppose this is not the case, so let $\alpha < \kappa^+$ be minimal for which there is $\beta < \alpha$
with $X_{\alpha, \beta} \notin \ee$, and having fixed $\alpha$, let $\beta < \alpha$ be minimal such that $X_{\alpha, \beta} \notin \ee$. 
For the remainder of this step we fix this choice of $\alpha, \beta$. 

Since $\beta < \alpha$, by construction (that is, by (e),(f),(g) of Step 2)
\[ X_{\alpha, \beta} \subseteq \{ \epsilon < \kappa : \beta \in u_{\alpha, \epsilon} \setminus w_{\alpha, \epsilon} \} \]

Define a function $\mx: \kappa \rightarrow \kappa$ by
\begin{eqnarray*}
 t \mapsto \max \{ \epsilon \leq t :   
& \langle f_\gamma(t) : \gamma \in u_{\alpha, \epsilon} \rangle ~\mbox{is $<_{L(M_1)}$-increasing}, \\
&  \langle g_\gamma(t) : \gamma \in u_{\alpha, \epsilon}\rangle ~\mbox{is $<_{L(M_1)}$-decreasing}, \\
& \mbox{and $\gamma \in u_{\alpha, \epsilon} \implies$}~ f_\gamma(t) <_{L(M_1)} g_\gamma(t) ~~\}   \\
\end{eqnarray*}

This is well defined by Step 1, condition (a): $0$ belongs to the set on the righthand side, and by continuity,
there are no new conditions at limits. 

For each $\epsilon < \kappa$, the set $\{ t < \kappa : \mx(t) > \epsilon \} \in \ee$. This is because:
\begin{enumerate}
\item by (c) $|u_{\alpha, \epsilon}| < \kappa$ 
\item by minimality of $\alpha$, for any $\gamma < \gamma^\prime < \alpha$ (e.g. any two elements of $u_{\alpha, \epsilon}$) 
we have that $X_{\gamma, \gamma^\prime} \in \ee$
\item $\ee$ is $\kappa$-complete  
\item by (a) $\epsilon^\prime < \epsilon \implies u_{\alpha, \epsilon^\prime} \subseteq u_{\alpha, \epsilon}$
\end{enumerate}

Notice that for any $t < \kappa$, $\mx(t) = t$ implies $u_{\alpha, t} = w_{\alpha, t}$. 
So if $\mx(t) = t$ on an $\ee$-large set, $X_{\alpha, \beta} \in \ee$, which would finish the proof. Suppose, then, that 
$Y = \{ t < \kappa : \mx(t) < t \} \in \ee$. By normality (Fact \ref{normal-fodor}), there is $Z \subseteq Y$, $Z \in \ee$ 
on which $\mx(t) = \epsilon_*$ for some fixed $\epsilon_* < \kappa$. But this contradicts the first sentence of the previous paragraph. 

These contradictions prove that for no $\alpha, \beta$ can it happen that $X_{\alpha, \beta} \notin \ee$, which finishes the proof 
of Step 3.

\step{Step 4: The pre-cut is not realized, i.e. it is indeed a cut.}
In this step we assume, for a contradiction, that there is $h \in {^\kappa{M_1}}$ such that for each $\alpha < \kappa^+$
\[ f_\alpha/\ee <_L h/\ee <_L g_\alpha/\ee  \]
i.e. $h$ realizes the type. Fixing such an $h$, let 
\[ A_\alpha = \{ \epsilon < \kappa : f_\alpha(\epsilon) <_L h(\epsilon) <_L g_\alpha(\epsilon) \} \in \ee \]
Since to each $\alpha$ we may associate a choice of index in $A_\alpha$, by Fodor's lemma for some $\epsilon_* < \kappa$, 
\[ S_1 = \{ \delta : \delta < \kappa^+, ~ \cf(\delta) = \kappa, ~ \epsilon_* \in A_\delta \} \]
is stationary. Furthermore, since $|u_{\epsilon_*, \alpha}| < \kappa$, there is some $w_* \subseteq u_{\epsilon, \alpha}$ 
for which 
\[ S_2 = \{ \delta \in S_1 : w_{\epsilon_*, \delta} = w_* \} \subseteq \kappa^+ \]
is stationary. Let $\delta_* \in S_2$ be such that $|\delta_* \cap S_2| = \kappa$. As 
$|w_{\epsilon_*, \delta_*}| \leq |u_{\epsilon_*, \delta_*}| < \kappa$, we may choose 
$\gamma_* \in S_2 \cap \{ \delta_* \setminus w_{\epsilon_*, \delta_*} \}$.

Now $w_{\epsilon_*, \delta_*} = w_{\epsilon_*, \gamma_*} = w_*$ since $\delta_*, \gamma_* \in S_2$, and note $\gamma_* < \delta_*$.
The definition of the sets $w$ (here, $w_*$) 
and Step 3, condition (e) means that when choosing $f_{\delta_*}(\epsilon_*), g_{\delta_*}(\epsilon_*)$
we would have ensured that 
\[ \beta \in w_* \implies f_{\beta}(\epsilon_*) < f_{\delta_*}(\epsilon_*) < g_{\delta_*}(\epsilon_*)
< g_{\beta}(\epsilon_*) \]
and likewise that
\[ \beta \in w_* \implies f_{\beta}(\epsilon_*) < f_{\gamma_*}(\epsilon_*) < g_{\gamma_*}(\epsilon_*)
< g_{\beta}(\epsilon_*) \]
On the other hand, $\gamma_* < \delta_*$, and $\gamma_* \notin w_*$. So when choosing $f_{\gamma_*}(\epsilon),
g_{\gamma_*}(\epsilon)$, Step 3, condition (g) would have meant we chose the intervals
$[f_{\delta_*}(\epsilon_*), g_{\delta_*}(\epsilon_*)]_{L(M_1)}$, 
$[f_{\gamma_*}(\epsilon_*), g_{\gamma_*}(\epsilon_*)]_{L(M_1)}$ to be disjoint. 

But we also know that $\gamma_*, \delta_* \in S_1$, so $h(\epsilon_*)$ must belong to both intervals. 
This contradiction completes Step 4 and the proof.
\end{proof}

\begin{rmk}
We know that if $D$ is any ultrafilter on $\kappa$ and $M$ is a model whose theory is not simple, then $M^\kappa/\de$ is not
$\kappa^{++}$-saturated. Still, Claim \ref{kappa-plus-cut} gives more precise information about the size of the cut:
we are guaranteed a cut of type $(\kappa^+, \kappa^+)$ as opposed to 
e.g. $(\kappa^+, \kappa)$. On the importance of symmetric cuts, see \cite{treetops}.
\end{rmk}

\begin{claim}
Assume $\kappa$ measurable, $\ee$ a $\kappa$-complete filter on $\kappa$, 
$\lambda \geq \kappa$,
$M_1$ a $\lambda$-saturated model with ${(L_M, <_M)}$ a dense linear order. Let $M_2 = M_1^\kappa/\ee$. Then
\end{claim}

\begin{proof}
Suppose for a contradiction that there were such a cut given by
$\langle f_\alpha : \alpha < \theta \rangle$, $\langle g_\beta : \beta < \sigma \rangle$ with
$\alpha_1, \alpha_2 < \theta, \beta_1, \beta_2 < \sigma \implies 
\{ t : M_1 \models f_{\alpha_1}(t) <_L f_{\alpha_2}(t) <_L g_{\alpha_2}(t) <_L g_{\alpha_1}(t) \} \in \ee$.
 Expand the language to add constants $\{ c_\alpha : \alpha < \theta \}$ where in the $t$th copy of the index model $M_1$,
 denoted $M_1[t]$,
 $c_\alpha$ is interpreted as $f_\alpha(t)$. Then in the ultrapower (which in the expanded language is an ultraproduct), 
 $\langle c_\alpha : \alpha < \theta \rangle$ forms the lower half of the supposed cut. For each $\beta < \sigma$, the set
 \[  A_\beta := \bigcap \{ \{ t ~:~ M_1[t] \models c_\alpha <_L g_{\beta}[t] \}  : \alpha < \theta \}  \in \ee \]
by $\kappa$-completeness. 

But recall that $M_1$ is a $\lambda$-saturated model, and $\sigma < \lambda$. Since for each $t$, we have
\[ | \{ \beta < \sigma : t \in A_\beta \} | \leq \sigma < \lambda \]
we may choose a new element $h \in {^IM_1}$ so that for each $t$, $h(t)$ satisfies
$\alpha < \theta \implies M_1[t] \models c_\alpha <_L h(t)$ and $t \in A_\beta \implies M_1[t] \models h(t) <_L g_\beta(t)$. 
By \lost theorem $h$ realizes our cut, which is the desired contradiction. 
\end{proof}

In a forthcoming paper the authors have shown that: 

\begin{thm-lit} \emph{(Malliaris and Shelah \cite{treetops})} \label{treetops-thm-b}
If $\de$ is a regular ultrafilter on $\lambda$ which saturates some theory with $SOP_2$, and
$M$ is a model of linear order, then among other things: 
\begin{enumerate}
\item for all $\mu \leq \lambda$ $M^\lambda/\de$ has no $(\mu, \mu)$-cut,
\item for all $\mu \leq \lambda$ there is at most one $\rho \leq \lambda$ such that $M^\lambda/\de$ has
a $(\mu, \rho)$-cut
\end{enumerate}
\end{thm-lit}

\begin{concl} \label{flex-not-good-tp2}
Let $\kappa < \lambda$ and suppose $\kappa$ is measurable. Then there exists a regular ultrafilter $\de$ on $I$, $|I| = \lambda$
which is flexible but not good, specifically not good for any theory with $SOP_2$.
\end{concl}

\begin{proof}
Let $D$ be a $\lambda^+$-good, $\lambda$-regular ultrafilter on $\lambda$. Let $\ee$ be a normal $\kappa$-complete, not 
$\kappa^+$-complete ultrafilter on $\kappa$. Let $\de = D \times \ee$ be the product ultrafilter. 
Then $\de$ is flexible by Claim \ref{reg-transfer}. On the other hand, by Claim \ref{kappa-plus-cut},
any $\de$-ultrapower of linear order will omit a $(\kappa^+, \kappa^+)$-cut. By Theorem \ref{treetops-thm-b},
$\de$ cannot saturate any theory with $SOP_2$.
\end{proof}

\begin{rmk}
On one hand, the advantage of Conclusion \ref{flex-not-good-tp2} over Theorem \ref{flexible-not-good-2} is in the greater range of cardinals:
we ask only that $\kappa < \lambda$, not $2^\kappa \leq \lambda$. On the other hand, Theorem \ref{flexible-not-good-2}
gives an a priori stronger failure of goodness, since the random graph is minimum among unstable theories in Keisler's order.
\end{rmk}

\section{Finite alternations of symmetric cuts} \label{s:finite-alt}

In this section we iterate the results of \S \ref{s:kappa-cuts}, \S \ref{s:kappa-plus-cuts} to produce regular ultrafilters $\de$ whose
library of cuts, $\mc(\de)$, contains any fixed finite number of alternations (or gaps).
The following definition is stated for regular ultrafilters only so that the choice of index model will not matter.

\begin{defn} \label{defn-alt-cuts}
Let $\kappa$ be a cardinal. Say that the regular ultrafilter $\de$ on $\lambda \geq \aleph_0$ has \emph{$\kappa$ alternations of cuts}
if there exist cardinals $\langle \mu_\ell : \ell < \kappa \rangle$, $\langle \rho_\ell : \ell < \kappa \rangle$ such that:
\begin{itemize}
\item $\ell_1 < \ell_2 < \kappa \implies \aleph_0 < \rho_{\ell_1} < \mu_{\ell_1} < \rho_{\ell_2} < \mu_{\ell_2} < \lambda$
\item for each $0 \leq \ell < \kappa$, $(\rho_\ell, \rho_\ell) \in \mc(\de)$, i.e.
$(\mathbb{N}, <)^\lambda/\de$ has some $(\rho_\ell, \rho_\ell)$-cut
\item for each $0 \leq \ell < \kappa $, $(\mu_\ell, \mu_\ell) \notin \mc(\de)$, i.e.
$(\mathbb{N}, <)^\lambda/\de$ has no $(\mu_\ell, \mu_\ell)$-cut
\end{itemize}
\end{defn}

We will start by proving a theorem for products of complete ultrafilters, Theorem \ref{alt-cuts},
and then extend it to regular ones in Theorem \ref{alt-regular} by adding one more iteration of the ultrapower.

First we observe that taking ultrapowers will not fill symmetric cuts whose cofinality is larger than the size of the index set.

\begin{obs} \label{not-filling-cuts}
Suppose $M$ is a $\lambda$-saturated model of linear order, $\kappa < \lambda$, $D$ an ultrafilter on $\kappa$.
If $M$ contains a $(\kappa_*, \kappa_*)$-cut, where $\kappa_* = \cf(\kappa_*) > \kappa$, 
then $M^\kappa/D$ will also contain a $(\kappa_*, \kappa_*)$-cut. More precisely, the image of the cut from 
$M$ under the diagonal embedding will remain unrealized in $M^\kappa/D$.
\end{obs}

\begin{proof}
Let the cut in $M$ be given by $(\langle f_\alpha : \alpha < \kappa_* \rangle, \langle g_\beta : \beta < \kappa_* \rangle)$,
and we consider the pre-cut given by $(\langle f_\alpha/D : \alpha < \kappa_* \rangle, \langle g_\beta/D : \beta < \kappa_* \rangle)$
in the ultrapower $M^\kappa/D$.
Suppose for a contradiction that there were a realization $h \in {^\kappa M}$. Let $\bx : \kappa_* \rightarrow \kappa$ be
a function which to each $\alpha < \kappa_*$ associates some index $\epsilon < \kappa$ for which 
$M \models f_\alpha(\epsilon) < h(\epsilon) < g_\alpha(\epsilon)$. By Fodor's lemma, there is a stationary subset $X \subseteq \kappa_*$
on which $\bx$ is constant and equal to, say, $\epsilon_*$. Then in $M$,
$(\langle f_\alpha(\epsilon_*) : \alpha \in X \rangle, \langle g_\beta(\epsilon_*): \beta \in X \rangle)$ will be cofinal
in the original cut, but by choice of $X$ it will be realized by $h(\epsilon_*)$, contradiction.
\end{proof}

Since the proof of Theorem \ref{alt-cuts} involves an inductive construction, it will be convenient to index the cardinals $\kappa_\ell$
in reverse order of size.

\begin{theorem} \label{alt-cuts} 
Suppose that we are given: 
\begin{enumerate}
\item[(a)] $n<\omega$ and $\kappa_{n} < \dots < \kappa_0 <  \kappa_{-1} = \lambda$
\item[(b)] $\ee_\ell$ a normal $\kappa_\ell$-complete ultrafilter on $\kappa_\ell$, for $\ell \leq n$. 
\item[(c)] $M_0$ a $\lambda$-saturated model which is, or contains, a dense linear order $<$
\item[(d)] $M_{\ell + 1} = (M_\ell)^{\kappa_\ell}/\ee_\ell$ for $\ell \leq n$ 
\end{enumerate}
Then:
\begin{enumerate}
\item[($\alpha$)] for $\ell \leq n$, $M_{\ell + 1}$ is $\kappa_\ell$-saturated
\item[($\beta$)] if $\ell < k \leq n+1$ then $M_k$ has a $({\kappa_\ell}^+, {\kappa_\ell}^+)$-cut
\item[($\gamma$)] for $i, \ell \leq n+1$, $M_\ell$ has \emph{no} $(\kappa_i, \kappa_i)$-cut 
\item[($\delta$)] for $\ell \leq n+1$, $M_\ell$ has no $(\theta, \theta)$-cut for $\theta < \lambda$ weakly compact
\end{enumerate}
Thus, for each $\ell \leq n$, 
$({\kappa_\ell}^+, {\kappa_\ell}^+) \in \mc(M_{n+1})$, and for any weakly compact $\theta < \lambda$, in particular $\theta = \kappa_\ell$,
$(\theta, \theta) \notin \mc(M_{n+1})$.
\end{theorem}

\begin{proof}
The ``thus'' clause summarizes ($\alpha$)-($\delta$). Recall that for transparency all languages are countable.
Note that condition (b) implies the cardinals $\kappa_\ell$ are measurable cardinals, thus limit cardinals, so
condition ($\gamma$) can never contradict condition ($\beta$).

($\alpha$) By induction on $-1 \leq \ell \leq n$ we verify that $M_{\ell + 1}$ is $\kappa_\ell$-saturated.
For $\ell = -1$, $M_0$ is $\lambda$-saturated and $\lambda > \kappa_0$. For $\ell > -1$, use the fact that
$M_{\ell+1} = ({M_\ell})^{\kappa_\ell}/\ee_\ell$ thus $M_{\ell+1} \equiv_{L_{\infty, \kappa_\ell}} M_\ell$
by \lost theorem for $L_{\kappa_i, \kappa_i}$. 

($\beta$) By Claim \ref{kappa-plus-cut} and Observation \ref{not-filling-cuts}. 

($\gamma$) Follows from ($\delta$) as measurable implies weakly compact. 

($\delta$) We prove this by induction on $\ell \leq n$. For $\ell = -1$, $M_0$ is $\lambda$-saturated.
For $\ell > -1$, let $\theta < \lambda$ be given and suppose we have a pre-cut in $M_{\ell+1}$
given by $(\langle f_\alpha : \alpha < \theta \rangle, \langle g_\beta : \beta < \theta \rangle)$.
There are three cases. If $\theta < \kappa_\ell$, then use ($\alpha$). If $\theta = \kappa_\ell$, use
Claim \ref{complete-cuts1}. So we may assume $\kappa_\ell < \theta$. Since $\theta$ is weakly compact, therefore inaccessible,
${2^{\kappa_\ell}} < \theta$. For $\alpha < \beta < \theta$ let 
\[ A_{\alpha, \beta} = \{ \epsilon < \kappa_\ell : f_\alpha(\epsilon) < f_\beta(\epsilon) < g_\beta(\epsilon) < g_\alpha(\epsilon) \} \in \ee_\ell \]
As $\theta$ is weakly compact, by Fact \ref{wcc-fact} the function $\bx : \theta \times \theta \rightarrow {2^{\kappa_\ell}} < \theta$
is constant on some $\uu \in [\theta]^{\theta}$. Call this constant value $A_*$. Now for $\epsilon \in A_*$, the sequence
$(\langle f_\alpha(\epsilon) : \alpha \in \uu \rangle, \langle g_\beta(\epsilon) : \beta \in \uu \rangle)$ is a pre-cut 
in $M_\ell$, meaning that $\alpha < \beta \in \uu \implies f_\alpha(\epsilon) < f_\beta(\epsilon) < g_\beta(\epsilon) < g_\alpha(\epsilon)$.

Let $B_* = \{ \epsilon \in A_* : ~\mbox{in $M_\ell$ there is $c$ such that $\alpha \in \uu \implies
f_\alpha(\epsilon) <_{M_\ell} c <_{M_\ell} g_\alpha(\epsilon)$} \}$. Now if
$A_* \setminus B_* \neq \emptyset$, for any $\epsilon \in A_* \setminus B_*$ we have that
$(\langle f_\alpha(\epsilon) : \alpha \in \uu \rangle, \langle g_\beta(\epsilon) : \beta \in \uu \rangle)$ 
is not just a pre-cut but also a cut in $M_\ell$, contradicting the inductive hypothesis. 
Thus for every $\epsilon \in A_*$ we may choose a realization $c(\epsilon)$ of the relevant pre-cut.
For $\epsilon \in \kappa \setminus A_*$, let $c(\epsilon)$ be arbitrary. 
Then $\langle c(\epsilon) : \epsilon < \kappa_\ell\rangle/\ee_\ell \in M_{\ell+1}$ realizes
$(\langle f_\alpha : \alpha < \theta \rangle, \langle g_\beta : \beta < \theta \rangle)$, as desired.
\end{proof}

By appending the construction of Theorem \ref{alt-cuts} to a suitable regular ultrafilter, we may produce
regular ultrapowers with $n$ alternations of cuts for any finite $n$.

\begin{theorem} \label{alt-regular}
Let $\lambda$ be an infinite cardinal, $n<\omega$ and suppose that there exist measurable cardinals
$\kappa_n < \dots < \kappa_0 <\lambda$.
Then there is a regular ultrafilter $\de$ on $I$, $|I| = \lambda$ such that:
\begin{enumerate}
\item $({\kappa_\ell}^+, {\kappa_\ell}^+) \in \mc(\de)$ for $\ell \leq n$
\item $(\theta, \theta) \notin \mc(\de)$ for $\theta < \lambda$ weakly compact, in particular $\ell \leq n$, $\theta = \kappa_\ell$
\item $\de$ is ${\kappa_n}^+$-good
\item $\de$ is $\lambda$-flexible 
\item $\de$ is $\lambda^+$-good for countable stable theories
\item $\de$ is not $(2^{\kappa_n})^+$-good for unstable theories
\end{enumerate}
\end{theorem}

\begin{proof}
Let $\de_1$ be a $\lambda$-regular, $\lambda^+$-good ultrafilter on $\lambda$. 
Let $\ee$ be the ultrafilter on $\kappa_n$ given by $\ee_0 \times \dots \times \ee_n$, where the $\ee_\ell$ are as in the statement
of Theorem \ref{alt-cuts}. Let $\de = \de_1 \times \ee$. We will show that $\de$ has the desired properties.

\begin{enumerate}
\item[(1)-(2)] This follows from having chosen $M_0$ in Theorem \ref{alt-cuts} to be a $\de_1$-ultrapower, thus $\lambda^+$-saturated
by the $\lambda^+$-goodness (and regularity) of $\de_1$.

\item[(3)] By Observation \ref{complete-good}.

\item[(4)] By induction on $\ell \leq n$, using Claim \ref{reg-transfer}(1) and the completeness of $\ee_\ell$, $\ell \leq n$.

\item[(5)] By induction on $\ell \leq n$, using Claim \ref{reg-transfer}(2). Since $\de_1$ is regular and $\lambda^+$-good, $\lcf(\aleph_0, \de) \geq \lambda^+$. 

\item[(6)] By Claim \ref{random-graph-saturation}.
\end{enumerate}
\end{proof}

\begin{qst}
Can Theorem \ref{alt-regular} be generalized to any number of alternations, not necessarily finite?
\end{qst}

\section*{Appendix} \label{s:normal}

In this appendix, we collect several easy proofs or extensions of proofs relevant to the material in the paper.

\begin{obs} \label{complete-good}
If $\ee$ is a $\kappa$-complete ultrafilter on $\kappa$ then $\ee$ is $\kappa^+$-good.
\end{obs}

\begin{proof} We adapt the proof of \cite{Sh:c} Claim 3.1 p. 334 that any ultrafilter is $\aleph_1$-good.
Suppose we are given a monotonic function $f : \fss(\kappa) \rightarrow \ee$. Define $g(u) = \bigcap \{ f(w) : \max{w} \leq \max{u} \}$. 
For each $u \in \fss(\kappa)$, $\max(u) < \kappa$ and moreover the number of possible finite $w \subseteq \{ \gamma : \gamma < \max(u) \}$
is $< \kappa$. Thus by $\kappa$-completeness, $g(u) \in \ee$. Clearly $g$ refines $f$, and 
since $\max(u \cup v) = \max\{ \max(u), \max(v) \}$, $g$ is multiplicative as desired. 
\end{proof}

Note, however, that while a good regular ultrafilter produces saturated ultrapowers, this need not be the case when the ultrafilter is complete, 
\emph{unless} the index models are also saturated. That is, 
the goodness of the ultrafilter $D$ on $I$ is equivalent to saturation of the ultrapower $M^I/D$ when
(a) we have regularity of the ultrafilter, \emph{or} (b) we assume the model $M$ is saturated.

\begin{fact} \label{fact-good}
Let $D$ be an ultrafilter on $I$ and $\lambda$ a cardinal. 
Then the following are equivalent:
\begin{enumerate}
\item $D$ is $\lambda$-good
\item for any model $M$ in a countable signature which is $\lambda$-saturated, $M^I/D$ is $\lambda$-saturated
\end{enumerate}
and if $D$ is regular, \emph{(1)} is equivalent to:
\begin{enumerate}
\item[(3)] for any model $M$ in a countable signature, $M^I/D$ is $\lambda$-saturated
\end{enumerate}
\end{fact}

\begin{proof}
See \cite{Sh:c} VI.2 in particular Theorems 2.2-2.3, Claim 2.4 and Lemma 2.11.
\end{proof}

By re-presenting the proof of \cite{Sh:c} Theorem VI.4.8 p. 379 to emphasize the role of incompleteness, we obtain a more general result.

\begin{theorem} \emph{} \label{lcf-omega}
Let $M$ be a $\lambda^+$-saturated model of an unstable theory $T$, $\vp$ an unstable formula and $E$ a $\kappa$-complete, $\kappa^+$-incomplete
ultrafilter on $\lambda$. 
Let $\delta = \lcf(\kappa, \ee)$. Then $M^I/\ee$ is not $(\kappa + \lcf(\kappa, \ee))^+$-saturated for $\vp$-types.
\end{theorem}

\begin{proof}
First consider the countably incomplete case.
We build a correspondence between a $\vp$-type and a $<$-type in an expanded language. 

Let $\vp = \vp(x;y)$, where
without loss of generality $\ell(x)=1$ but $\ell(y)$ need not be $1$. Choose, for each $m < n <\omega$, sequences
$\overline{a}^n_m$ where $\ell(\overline{a}^n_m) = \ell(y)$, so that $\vp$ has the order property over each
$\langle \overline{a}^n_m : m < n \rangle$: i.e., $k < n < \omega$ implies $\{ \vp(x;\overline{a}^n_m)^{if (m > k)} : m < n \}$ is consistent.

Let $\langle b_n : n < \omega \rangle$ be a sequence of distinct elements. Let $P$ be a new unary relation symbol,
$<$ a new binary relation symbol, and for each $\ell < \ell(y)$ let $F_\ell$ be a new binary function symbol. 
Let $M_1$ denote the expansion of $M$ by these new symbols, as follows.
$P^{M_1} = \{ b_n : n < \omega \}$ and $<^{M_1} = \{ \langle b_k, b_n \rangle  : k < n < \omega \}$. Finally, interpret the
functions $F_\ell$ so that for each $k<n<\omega$, $\overline{a}^n_k = \langle F_0(b_k, b_n), \dots F_{\ell(y)-1}(b_k, b_n) \rangle$.

Now we take the ultrapower $N = M^\lambda/E$ and let $N_1$ denote the corresponding ultrapower in the expanded language. 
In $N_1$, $P$ is a linear order and so by the hypothesis on $E$ we have a $\delta$-cut over the diagonal embedding of the
sequence $\langle b_n : n < \omega \rangle$. Choose a sequence $\langle c_i : i < \delta \rangle$ witnessing this, 
so (1) each $c_i \in P^{N_1}$, (2) $n <\omega$ and $i<j<\delta$ implies $N_1 \models b_n < c_i < c_j$, 
(3) for no $c \in P^{N_1}$ is it the case that for each $n<\omega, i <\delta$, $N_1 \models b_n < c < c_i$. 

We now translate back to a $\vp$-type. Consider:
\[ q(x) = \{ \neg \vp(x,F_0(b_n, c_0), \dots F_{\ell(y)-1}(b_n, c_0)) : n < \omega \} \cup 
\{ \vp(x,F_0(c_i, c_0), \dots F_{\ell(y)-1}(c_i, c_0)) : 0 < i < \delta \} \]
This is a consistent partial $\vp$-type by $\lost$ theorem, since for $n < \omega, i < \delta$ we have that
$b_n < c_i \mod E$. We will show that $q$ is omitted. Suppose it were realized,
say by $a$. Let $\langle X_n : n < \omega \rangle$ be a sequence of elements of $E$ witnessing that $E$ is 
$\aleph_1$-incomplete. Let $\langle Y_n : n < \omega \rangle$ be a sequence of elements of $E$ given by
\[ Y_n = \{ t \in \lambda : M \models \bigwedge \{ \neg \vp(a[t],F_0(b_k, c_0[t]), \dots F_{\ell(y)-1}(b_k, c_0[t])) : k \leq n \} \]
which exists by \lost theorem. (We write $b_k$ rather than $b_k[t]$ since these are essentially constant elements. 
Note that in each index model, $c_0[t]$ is simply one of the $b_n$s.)

For each $t \in I$ define $\rho(t) = \min\{ n : t \notin X_{n+1} \cap Y_{n+1} \}$. By the assumption of $\aleph_1$-incompleteness,
$\rho$ is well defined. Define a new element $b \in {^\lambda M}$ by: $b[t] = b_{\rho(t)}$. By \lost theorem $P^{N_1}(b)$, 
so let us determine its place in the linear order $<^{N_1}$.

First, for each $n<\omega$, we have $b > b_n \mod E, as $E is $\omega$-complete (i.e. a filter) thus:
\[ \{ t : c[t] > b_n[t] \} \supseteq \bigcap_{j \le n} ( X_{j+1} \cap Y_{j+1} ) \in \de \]

On the other hand, suppose that for some $i < \delta$ we had $c_i \leq b \mod E$. Then by \lost theorem and definition of $\rho$,
it would have to be the case that
\[ N_1 \models \neg \vp(a[t],F_0(c_i[t], c_0[t]), \dots F_{\ell(y)-1}(c_i[t], c_0[t])) \]
contradicting the definition of $q$. So for each $i < \delta$, we have that $b < c_i \mod E$.

Thus from a realization $a \models q$ we could construct a realization $b$ of the $(\aleph_0, \delta)$-cut in $P^N$. 
Since the latter is omitted, $q$ must be as well, which completes the proof for countably incomplete filters.

Now for the general case: If $\lambda = \kappa$,
modify the argument of Theorem \ref{lcf-omega} by replacing $\omega$ with $\kappa$ everywhere in the proof just given,  
and \lost theorem by \lost theorem for $L_{\kappa, \kappa}$. 
If $\lambda > \kappa$, begin by choosing a surjective map $\bh : \lambda \rightarrow \kappa$ 
so that $\ee = h(E)$ is a nonprincipal ultrafilter on $\kappa$, thus $\kappa$-complete not $\kappa^+$-complete. 
\end{proof}

For completeness, we verify that $\lambda$-flexible corresponds to $\lambda$-OK, Definition \ref{d:ok} above.

\begin{obs} \label{flexible-ok} Suppose that $\de$ is an $\aleph_1$-incomplete ultrafilter on $I$.
Then the following are equivalent. 
\begin{enumerate}
\item $\de$ is $\lambda$-O.K.
\item $\de$ is $\lambda$-flexible.
\end{enumerate}
\end{obs}

\begin{proof}
(1) $\rightarrow$ (2) Let $M$ be given with $(\mathbb{N}, <) \preceq M$ and 
let $h_0 \in {^I M}$ be any $\de$-nonstandard integer.
Let $\{ Z_n : n < \omega \} \subseteq \de$ witness that $\de$ is $\aleph_1$-incomplete.
Without loss of generality, $n < n^\prime \implies Z_n \supseteq Z_{n^\prime}$.
Let $h_1 \in { ^I \mathbb{N} }$ be given by $h_1(t) = \max \{ n : t \in Z_n \}$.
Define $h \in {^I \mathbb{N}}$ by 
\[ h(t) = \min \{ h_0(t), h_1(t) \} \]
Then for each $n\in \mathbb{N}$, $X_n := \{ t : h(t) \geq n \} \in \de$ and $X_n \subseteq Z_n$,
thus $\bigcap \{ X_n : n \in \mathbb{N} \} = \emptyset$.

Define a function $f: \fss(\lambda) \rightarrow \de$ by $f(u) = X_{|u|}$. 
As $\de$ is $\lambda$-OK, we may choose $g$ to be a multiplicative refinement of $f$,
and consider $\mathbf{Y} = \{ Y_i : i < \lambda \}$ given by $Y_i = g(\{ i \})$.

First, we verify that $\mathbf{Y}$ is a regularizing family, by showing that each $t \in I$ can only belong
to finitely many elements of $\mathbf{Y}$. Given $t \in I$, 
let $m = h(t) + 1 < \omega$, so $t \notin X_m$. Suppose there were $i_1 < \dots < i_m < \lambda$ such that
$t \in g(\{ i_1 \}) \cap \dots \cap g(\{ i_m \})$. As $g$ is multiplicative and refines $f$, this would imply
$t \in g(\{ i_1, \dots i_m \}) \subseteq f(\{ i_1, \dots i_m \}) = X_m$, a contradiction. Thus $\mathbf{Y}$ is a regularizing family.
Moreover, as $t$ was arbitrary, we have shown that 
\[ \left| \{ i < \lambda : t \in g(\{ i \}) \} \right| \leq h(t) \leq h_0(t) \]
and thus that $\mathbf{Y}$ is a regularizing family below $h_0$.
As $h_0$ was an arbitrary nonstandard integer, this completes the proof.

(2) $\rightarrow$ (1) Let $f: \fss(\lambda) \rightarrow \de$ be such that $|u| = |v| \implies f(u) = f(v)$, 
and we will construct a multiplicative refinement for $f$. 
Let $\langle Z_n : n < \omega \rangle$ witness the $\aleph_1$-incompleteness of $\de$, and as before,
we may assume $n < n^\prime \implies Z_n \supseteq Z_{n^\prime}$. 
For each $t \in I$, let $\rho \in {^I\mathbb{N}}$ be given by 
$\rho(t) = \max \{ n \in \mathbb{N} : t \in f(n) \cap Z_n \}$, which is well defined by the choice of the $Z_n$.
Now for each $m \in \mathbb{N}$
\[ \{ t \in I : \rho(t) > m \} \supseteq \bigcap \{  f(n) \cap Z_n : n \leq m \} \in \de \]
so $\rho$ is $\de$-nonstandard. Applying the hypothesis of flexibility, let $\{ Y_i : i < \lambda \}$ be a $\lambda$-regularizing family
below $\rho$. Let $g: \fss(\lambda) \rightarrow \de$ be given by $g(\{ i \}) = f(\{i\}) \cap Y_i$ and for $|u| > 1$,
$g(u) = \bigcap \{ g(\{i\}) : i \in u \}$. Thus $g$ is multiplicative, by construction. Let us show that it refines $f$.
Given any $n<\omega$ and $i_1 < \dots < i_n <\lambda$, observe that by definition of ``below $\rho$'' we have
$t \in Y_{i_1} \cap \dots \cap Y_{i_n} \implies \rho(t) \geq n$. Applying this fact and the definitions of $g$ and $f$,
\[ g( \{ i_1, \dots i_n \} ) \subseteq \bigcap \{ Y_{i_j} : 1 \leq j \leq n \} \subseteq \{ t \in I : \rho(t) \geq n \} \subseteq f(n) = 
f(\{ i_1, \dots i_n \}) \]
thus $g$ refines $f$, which completes the proof.
\end{proof}

\newpage


\begin{thebibliography}{50}
\bibitem{b-k} J. Baker and K. Kunen. ``Limits in the uniform ultrafilters.'' 
Trans. Amer. Math. Soc. 353 (2001), no. 10, 4083--4093. 
%
%

\bibitem{ck73} C. C. Chang and H. J. Keisler, \emph{Model Theory}, North-Holland, 1973. 

\bibitem{dow} A. Dow, ``Good and OK ultrafilters.'' 
Trans. AMS, Vol. 290, No. 1 (July 1985), pp. 145-160.




%
\bibitem{keisler-1} H. J. Keisler, ``Good ideals in fields of sets,'' Ann. of Math. (2) 79 (1964), 338-359.

\bibitem{keisler} H. J. Keisler, ``Ultraproducts which are not saturated.'' J. Symbolic Logic 32 (1967) 23--46.


\bibitem{KSV} J. Kennedy, S. Shelah, and J. Vaananen. ``Regular Ultrafilters and Finite Square Principles.'' J. Symbolic Logic 73 (2008) 817-823.

%
%
%
\bibitem{kunen} K. Kunen, ``Ultrafilters and independent sets.'' Trans. Amer. Math. Soc. 172 (1972), 299--306. 
%
\bibitem{kunen-1} K. Kunen, ``Weak P-points in BN - N,'' Proc. Bolyai Janos Soc. Colloq. on Topology, 
1978, pp. 741-749.
%
\bibitem{los} J. \los, ``Quelques Remarques, Theoremes et Problemes sur les Classes Definissables d'Algebres,'' in 
Mathematical Interpretation of Formal Systems, Studies in Logic, 1955, 98-113.

\bibitem{mm1} M. Malliaris, ``Realization of $\vp$-types and Keisler's order.'' Ann. Pure Appl. Logic 157 (2009), no. 2-3, 220--224

\bibitem{mm-thesis} M. Malliaris, Ph. D. thesis, U.C. Berkeley (2009). Available at http://math.uchicago.edu/$\sim$mem.
\bibitem{mm4} M. Malliaris, ``Hypergraph sequences as a tool for saturation of ultrapowers.'' (2010) Accepted, JSL. %

\bibitem{mm5} M. Malliaris, ``Independence, order and the interaction of ultrafilters and theories.'' (2010) Accepted,
APAL. 
%
%
\bibitem{mm-sh-v2} M. Malliaris and S. Shelah, ``Model-theoretic properties of ultrafilters built by independent families of functions.'' (2012)

\bibitem{mm-sh3} M. Malliaris and S. Shelah, ``A dividing line within simple unstable theories.'' (2012)

\bibitem{treetops} M. Malliaris and S. Shelah, ``On the set of cofinalities of cuts in ultrapowers of linear order.'' (2012)


\bibitem{Sh:a} 
S. Shelah, \emph{Classification Theory}, North-Holland, 1978. 

\bibitem{Sh:c} 
S. Shelah, \emph{Classification Theory and the number of non-isomorphic models}, rev. ed., North-Holland, 1990. 

\bibitem{Sh93} S. Shelah, ``Simple unstable theories.'' Ann. Math. Logic, 19, 177--203 (1980)

\bibitem{Sh500} �
S. Shelah, ``Toward classifying unstable theories.'' Annals of Pure and Applied Logic 80 (1996) 229--255. 

\bibitem{ShUs}
S. Shelah and A. Usvyatsov, ``More on ${\rm SOP}_1$ and ${\rm SOP}_2$.'' Ann. Pure Appl. Logic 155 (2008), no. 1, 16--31. 
\end{thebibliography}
\end{document}